%% file: amoebaev2.tex
\documentclass[reqno]{amsart}
\usepackage{amssymb,amsmath}
\usepackage{verbatim}
\usepackage{mathrsfs}
\usepackage[all]{xy}
\usepackage[pdftex]{graphicx}
\usepackage[,pdftex]{hyperref}

\title{Degenerations of amoebae and Berkovich spaces}
\author{Mattias Jonsson}
\address{Dept of Mathematics\\
  University of Michigan\\
  Ann Arbor, MI 48109-1043\\
  USA}
\email{mattiasj@umich.edu}

\date\today
\thanks{Partially supported by NSF grants 
  DMS-1001740 and DMS-1266207.}


\include{preamb}
\begin{document}
\begin{abstract}
  We prove a continuity result for the fibers of the Berkovich 
  analytification of a complex algebraic variety with respect to
  the maximum of the Archimedean norm and the trivial norm.
  As a consequence, we obtain generalizations of a result 
  of Mikhalkin and Rullg{\aa}rd  about degenerations of 
  amoebae onto tropical varieties.
\end{abstract}

\maketitle
%
%
%
%
%
%
\section*{Introduction}
Let $X\subset(\C^*)^n$ be an algebraic subvariety of the
$n$-dimensional complex algebraic torus.
The \emph{amoeba} $A_X\subset\R^n$ of $X$ is the image of $X$
under the map 
\begin{equation*}
  \cL\colon(\C^*)^n\to\R^n
\end{equation*}
defined by\footnote{We use negative signs to match the standard
convention for valuations.  All logarithms are natural logarithms.} 
$\cL=(-\log|z_1|,\dots,-\log|z_n|)$,
where $(z_1,\dots,z_n)$ are coordinates on $(\C^*)^n$.
See Figure~\ref{F101} for a picture of the amoeba of
$X=\{z_1+z_2+1=0\}$.

More generally, let $(K,|\cdot|)$ be any complete 
valued field and let $X\subset K^{*n}$ be an algebraic variety.
For any valued field extension $L/K$, let 
$X_L\subset L^{*n}$ be the base change.
Define the \emph{tropicalization} of $X$ to be the subset 
$\Xtrop\subset\R^n$ defined by 
\begin{equation*}
  \Xtrop=\bigcup_L\cL(X_L),
\end{equation*}
where $L$ ranges over all valued field extensions of $K$
and $\cL\colon(L^*)^n\to\R^n$ is defined using the same 
formula as above.

For example, suppose $K=\C$.
If $|\cdot|_\infty$ is the usual Archimedean norm on $\C$,
then $(\C,|\cdot|_\infty)$ does not admit any 
nontrivial valued field extensions, so $\Xtrop=A_X$ in this case.
On the other hand, we can also equip $\C$ with the 
\emph{trivial norm} $|\cdot|_0$, for which $|a|_0=1$ for all $a\in\C^*$.
Then the tropicalization of
$X$ is equal to the cone over the \emph{logarithmic limit set} of $X$
introduced in~\cite{Bergman}.
The case $X=\{z_1+z_2+1=0\}$ is depicted to the right in Figure~\ref{F101}.
We see that the tropicalization looks like the 
large scale limit of the amoeba.
This is a general fact:
\begin{ThmA}
  The large scale limit of the amoeba $A_X$ equals the tropicalization of~$X$:
  \begin{equation*}
    \lim_{\rho\to{0+}}\rho\cdot A_X
    =\Xtrop,
  \end{equation*}  
  where the tropicalization is computed using the trivial norm 
  on $\C$. 
\end{ThmA}
Here 
$\rho\cdot A_X:=\{\rho\cdot v\mid v\in A_X\}$ for $\rho\in\R_+^*$
and the limit can be understood, for example, in the sense of
Kuratowski convergence. When $X$ is a hypersurface, 
Theorem~A is a special case of a result by Rullg{\aa}rd and 
Mikhalkin; see below. The general case of Theorem~A is 
proved in~\cite{Bergman} in a slightly different language and 
conditional on a conjecture that was later 
establied in~\cite{BieriGroves}.

\begin{figure}
  \includegraphics[width=4cm]{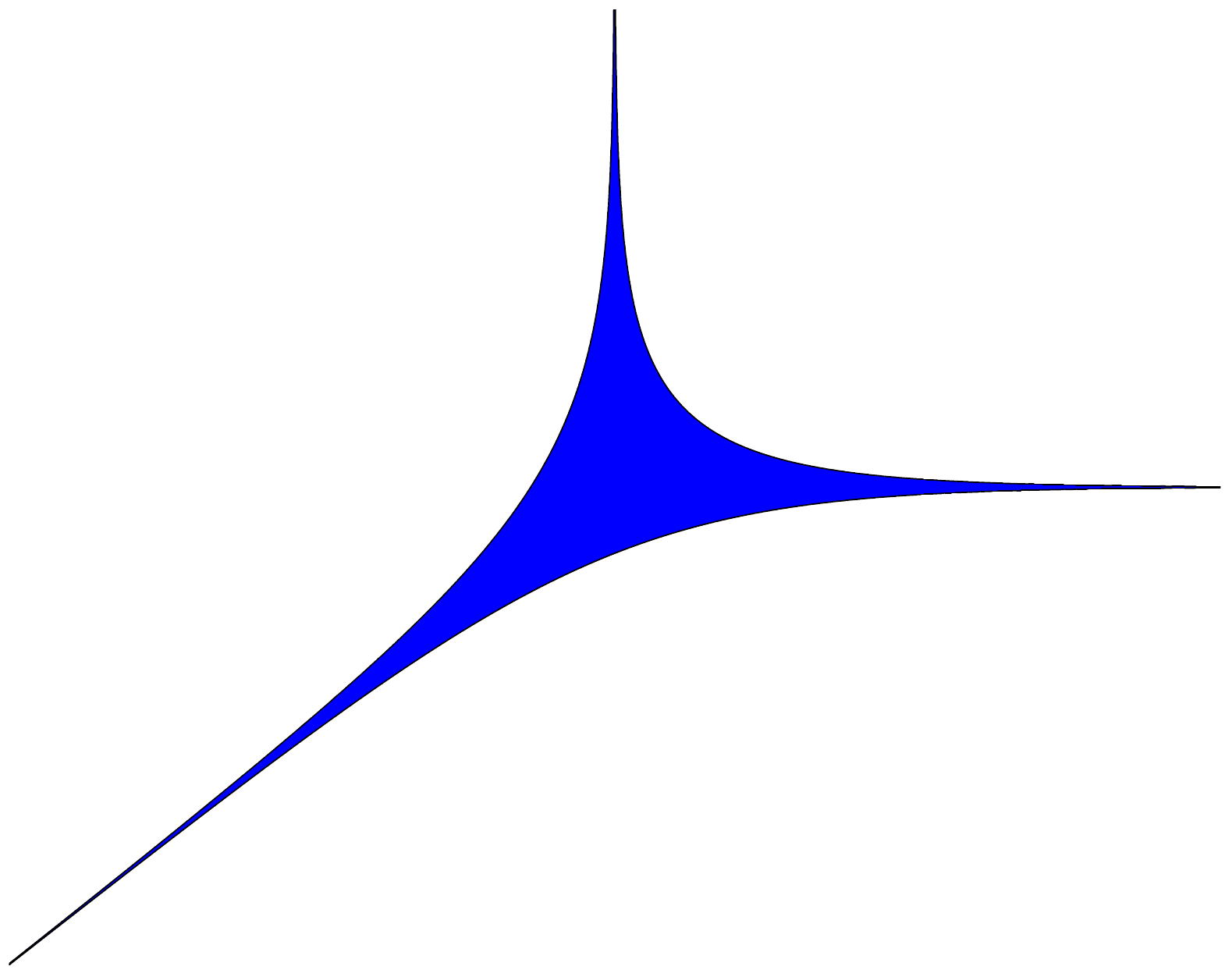}
  \hspace*{1cm}
  \includegraphics[width=4cm]{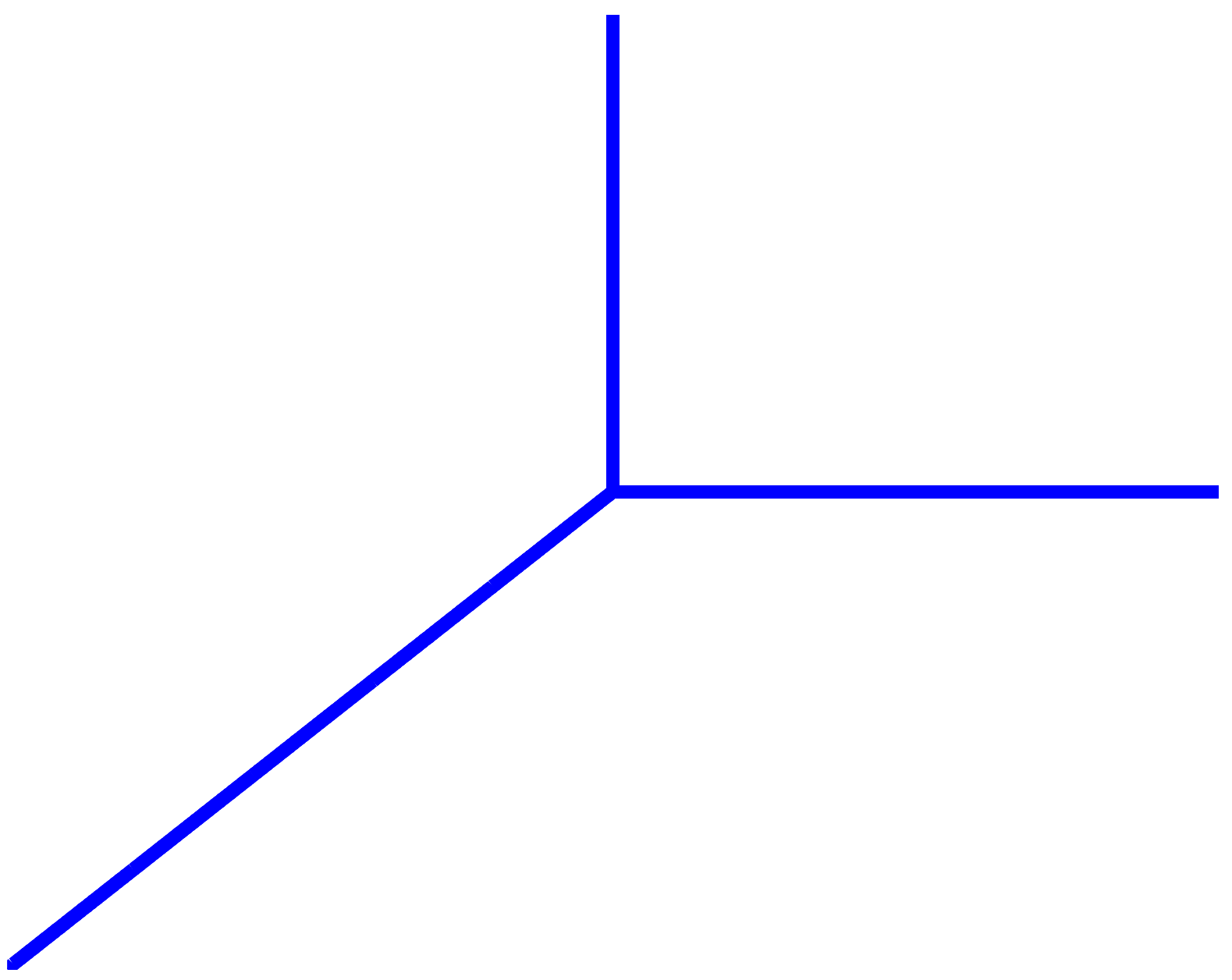}
  \caption{The amoeba and the tropicalization of the curve $z_1+z_2+1=0$
    in $\C^{*2}$.}\label{F101}
\end{figure}

As a more global version, consider a (complex) toric variety $Y$. 
There is a natural topological space $\Ytrop$ 
canonically associated to $Y$, see~\S\ref{S102}.
If $Y=\C^{*n}$, then $\Ytrop=\R^n$; in general $\Ytrop$
contains $\R^n$ as an open dense subset and comes with 
a multiplicative action by $\R_+^*$ extending the usual action on $\R^n$.

The two absolute values on $\C$ above define two different tropicalization
maps of $Y$ onto $\Ytrop$. If $X$ is an algebraic subvariety of $Y$,
let $A_X$ and $\Xtrop$ denote the images of $X$ in $\Ytrop$
under these two maps. When $Y$ is projective, $A_X$ is 
homeomorphic to the \emph{compactified amoeba} defined in~\cite{GKZ}.
\begin{ThmAp}
  We have $\lim_{\rho\to{0+}}\rho\cdot A_X=\Xtrop$.
\end{ThmAp}
Note that the notation is somewhat abusive since both $A_X$ and $\Xtrop$
depend on the embedding of $X$ in a toric variety $Y$.
Theorem~A is the special case of Theorem~A$'$
when $Y$ is the algebraic torus.

\smallskip
New we consider one-parameter families of subvarieties.
Let $\cX\subset\C^*\times Y$ be a closed algebraic subvariety 
such that the projection of $\cX$ onto the first factor 
$\C^*$ is surjective. 
Write $X_a\subset Y$ for the fiber of $\cX$ above $a\in\C^*$
and $A_{X_a}\subset\Ytrop$ for the amoeba as in Theorem~A$'$.
Define $\cXtrop$ as the tropicalization of the base change
$\cX\times_{\G_m}\Spec\C\lau{t}$, where 
$\G_m=\Spec\C[t^{\pm1}]\simeq\C^*$ and 
the field $\C\lau{t}$ of formal Laurent series is equipped
with the usual non-Archimedean absolute value for which $|t|=e^{-1}$.
\begin{ThmB}
  We have $\lim_{a\to0}(\log{|a|^{-1}})^{-1}\cdot A_{X_a}=\cXtrop$.
\end{ThmB}
See Figure~\ref{F103} for an illustration of Theorem~B in the case
$Y=\C^{*2}$ and $\cX=\{z_1+z_2+t=0\}\subset\C^*\times Y$,
and Figure~3 for the same situation with $Y=\P^2$.
When $\cX=\C^*\times X$ is a product, Theorem~B reduces to 
Theorem~A$'$.

\smallskip
Theorem~B is due to Rullg{\aa}rd~\cite[Thm.~9]{Rul01} and
Mikhalkin~\cite[Cor~6.4]{Mik04a} in the case when $Y=\C^{*n}$ and 
$\cX\subset\C^*\times\C^{*n}$ is a hypersurface;
see also~\cite[Thm.~7.1]{Tei08}. 
The proofs in \textit{loc.\ cit.}\ use the characterization of $\Xtrop$ as
the locus where the tropicalization of the Laurent polynomial defining 
$X$ fails to be affine.
The approach in~\cite{Mik04a} also emphasizes the analogy
with the patchworking construction of Viro~\cite{Viro}.
As these proofs show, 
the scaled amoebae in fact converge to the tropicalization in the 
Hausdorff metric; see also~\cite{AKNR13}.

The higher codimension case of Theorem~B for $Y=\C^{*n}$ 
is stated without proof in~\cite[Thm.~1.4]{IMS09}.
I have not been able to locate a proof nor the 
general version of Theorem~B in the literature.
At least in the case $Y=\C^{*n}$, one may in principle reduce Theorem~B to the
hypersurface case, using, on the one hand, Artin's Approximation
Theorem together with the approach in~\cite[p.112]{Tei08} 
and, on the other hand, the fact that the tropicalization of a
subvariety is the intersection of the tropicalizations of finitely
many hypersurfaces containing the subvariety. However, the latter fact is
quite nontrivial, with incomplete proofs appearing in the literature:
a correct argument can be found by 
combining~\cite{BJSST} and~\cite{CartwrightPayne}, or in~\cite{MS15}.

\begin{figure}
  \includegraphics[width=3.6cm]{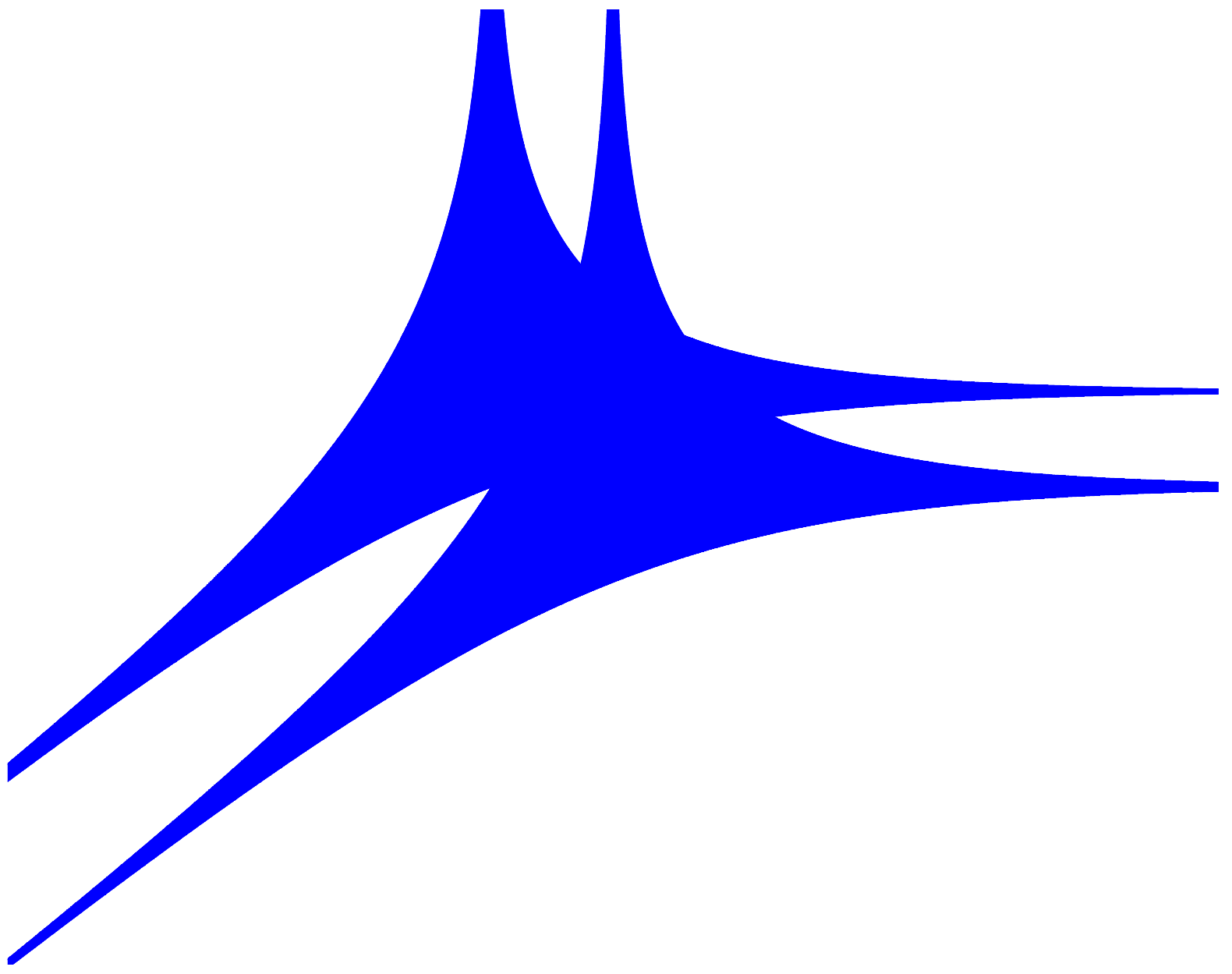}
  \includegraphics[width=3.6cm]{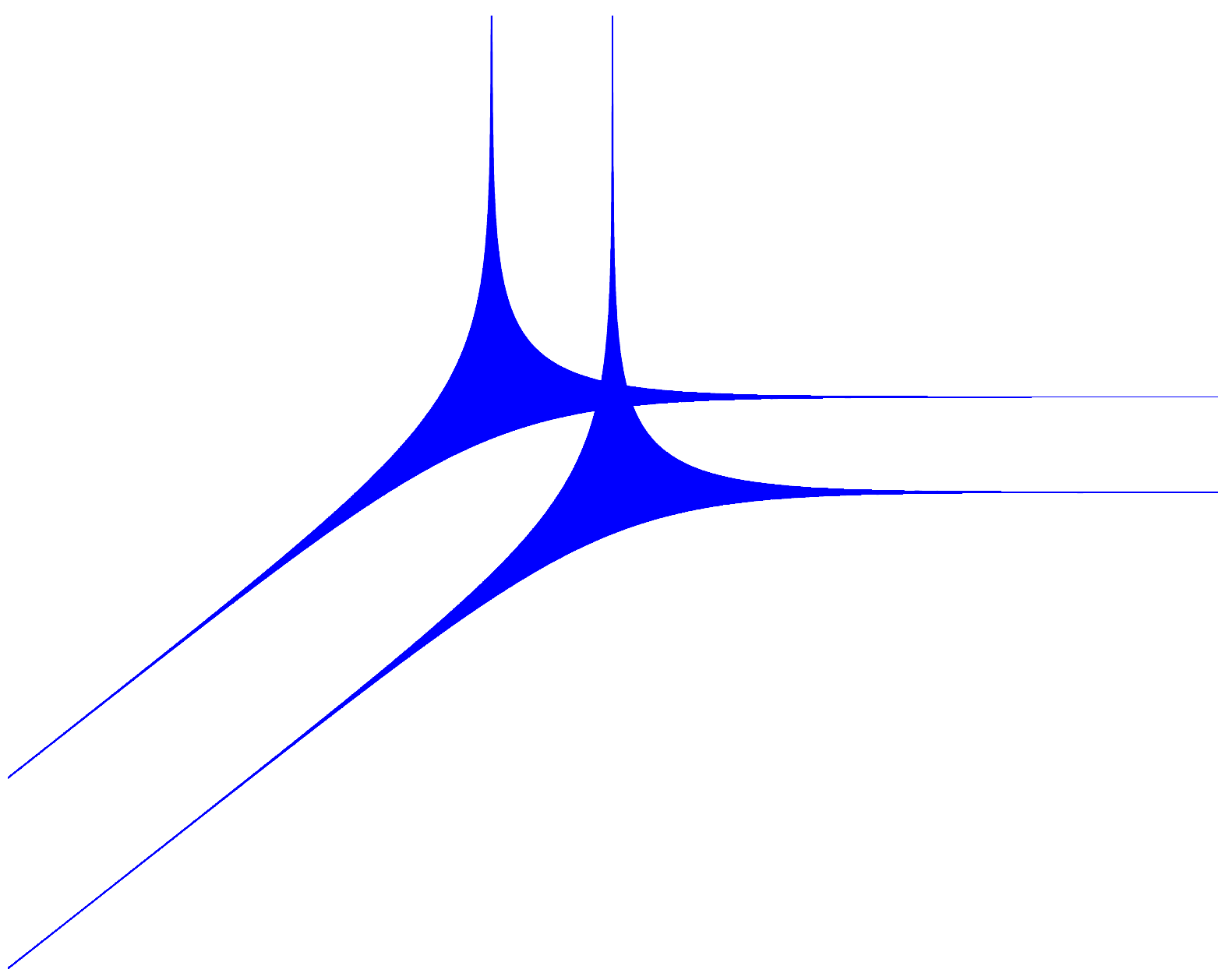}
  \includegraphics[width=3.6cm]{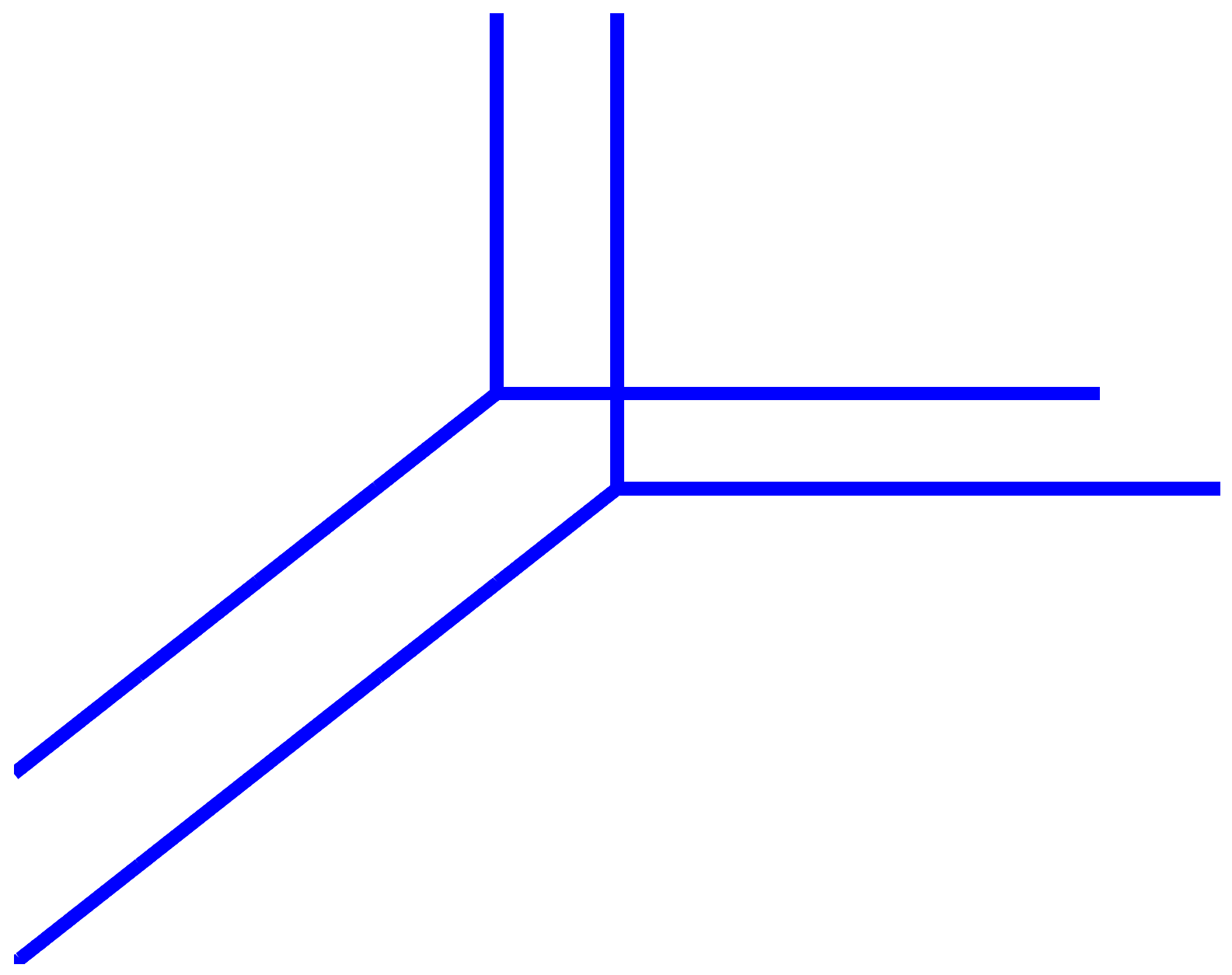}
  \caption{These pictures illustrating Theorem~B show two  
    scaled amoebae and the tropicalization of the curve 
    $V(f_t)$ in the torus $Y=\C^{*2}$, where $f_t=(z_1+z_2+1)(tz_1+t^{-1}z_2+1)$.
    The first two pictures show the amoeba of the complex
    curve $V(f_a)$, scaled by a factor
    $(\log|a|^{-1})^{-1}$, for $a=0.5$ and $a=0.2$, respectively.
    The last picture shows the tropicalization of the curve
    $V(f_t)$ over $\C\lau{t}$.}\label{F103}
\end{figure}

\begin{figure}
  \includegraphics[width=3.3cm]{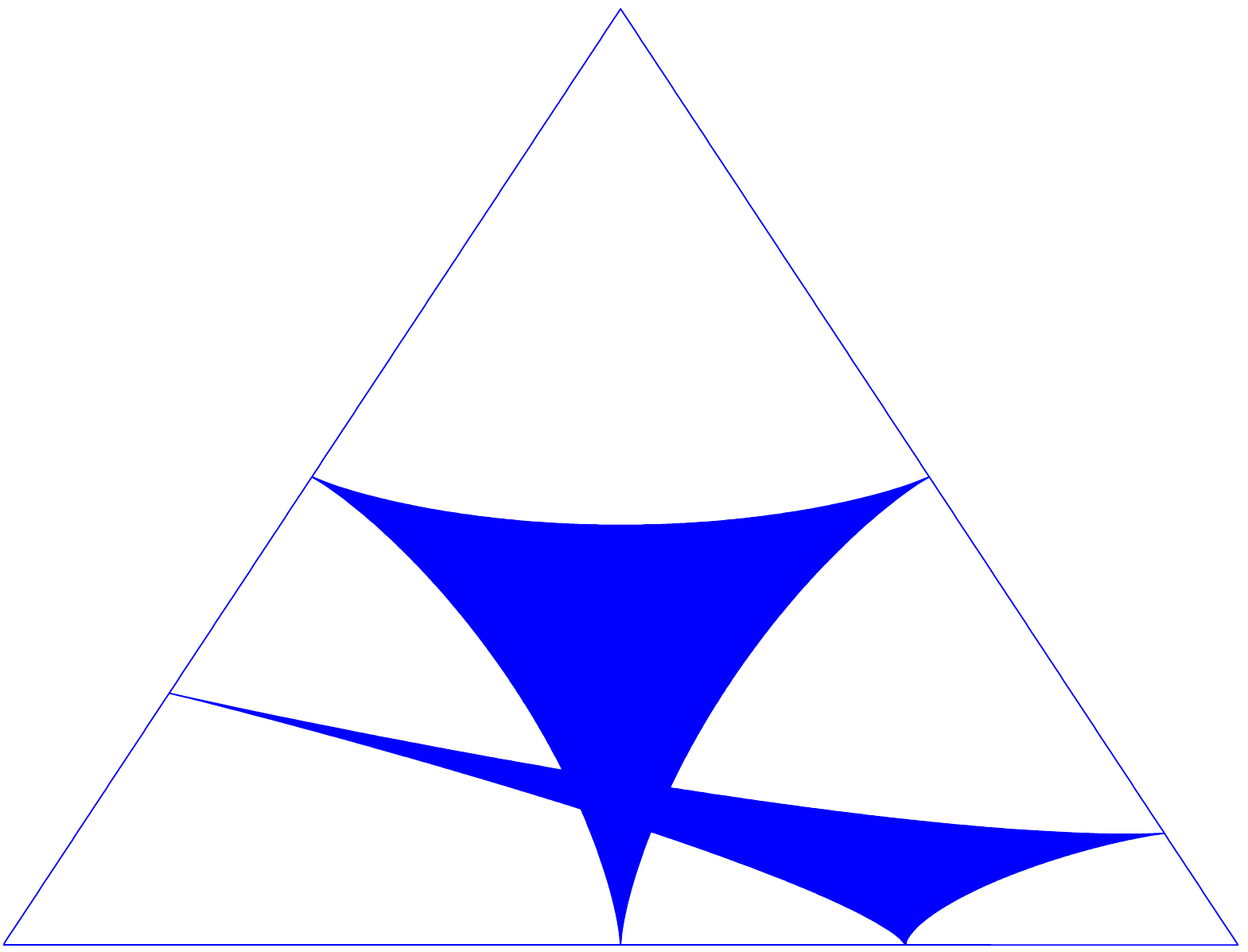}
  \hspace*{0.5cm}
  \includegraphics[width=3.3cm]{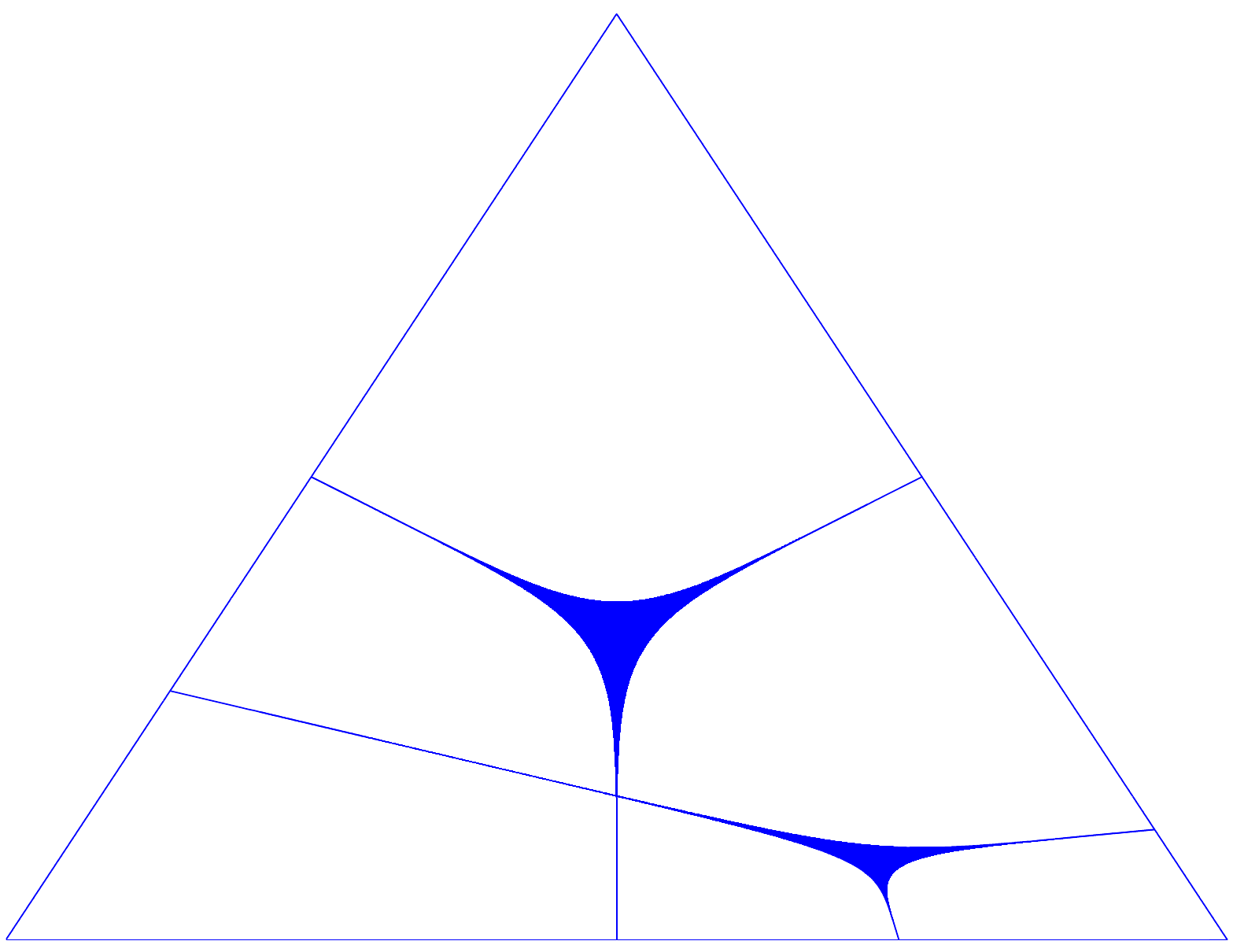}
  \hspace*{0.5cm}
  \includegraphics[width=3.3cm]{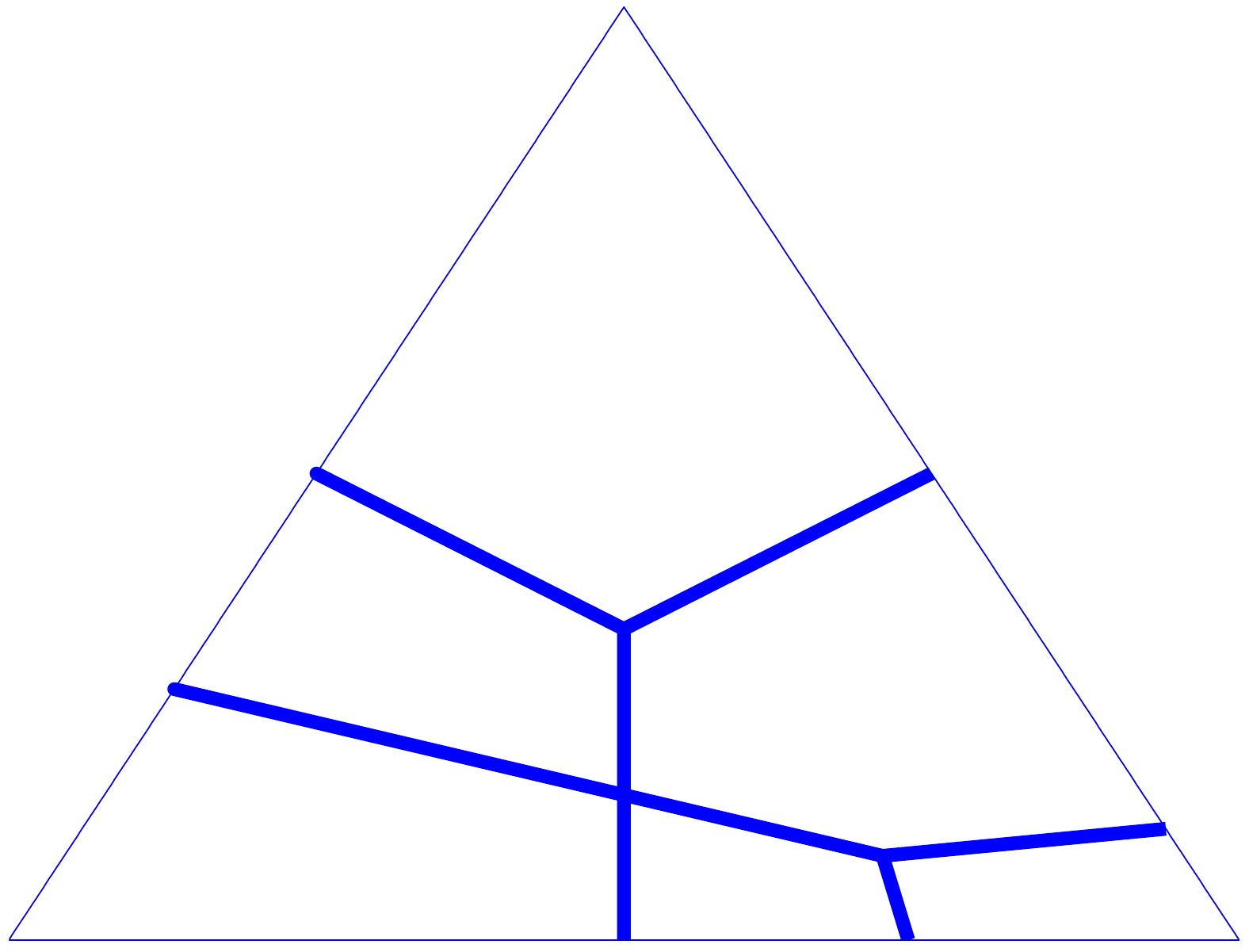}
  \caption{This picture illustrates Theorem~B for the closure
    in $\P^2$ of the curve $V(f_t)$ in Figure~\ref{F103}. 
    The triangle is the moment polytope of 
    $Y=\P^2$ with its canonical polarization.}\label{F104}
\end{figure}

\medskip
Our proof of Theorem~B is quite different and does not rely on
reduction to the hypersurface case. Indeed, 
the purpose of this paper is to show that these results on
degenerations of amoebae 
are rather direct consequences of a 
continuity property of the fibers of certain \emph{Berkovich spaces}
that were introduced in~\cite{BerkHodge} and contain both Archimedean
and non-Archimedean information.
Our results give further evidence to the suggestion on~p.51 of~\loccit
that such spaces are ``worth studying''.

\smallskip
Let us explain all this in the context of Theorem~A$'$,
leaving the setting of Theorem~B to~\S\ref{S104}.
Consider the field $\C$ equipped with the norm
\begin{equation}\label{eq:norm}
  \|\cdot\|:=\max\{|\cdot|_\infty,|\cdot|_0\},\tag{$\blacklozenge$}
\end{equation}
Note that $\|\cdot\|$ is only submultiplicative, but $(\C,\|\cdot\|)$ is
nevertheless a Banach ring. Given a complex algebraic variety 
$X$, Berkovich introduced in~\cite{BerkHodge} 
a natural \emph{analytification} $\XAn$
of $X$ with respect to the norm $\|\cdot\|$ on $\C$.
See~\S\ref{S101} for more details on this and on what follows.
The space $\XAn$  is a locally compact Hausdorff space and 
comes with a natural continuous and surjective map 
\begin{equation*}
  \lambda\colon\XAn\to[0,1].
\end{equation*}
The fiber $\lambda^{-1}(1)$ is the usual complex analytic space $X^h$
associated 
to $X$.\footnote{The superscript ``$h$'' stands for ``holomorphic''.}
For $0<\rho\le 1$, $\lambda^{-1}(\rho)$ is
homeomorphic to $X^h$. Finally, the fiber $\lambda^{-1}(0)$
is the Berkovich analytification of $X$ with respect to $|\cdot|_0$.
See Figure~\ref{F102} for an illustration of $(\P^1)^\mathrm{An}$.
\begin{figure}
  \includegraphics[width=12cm]{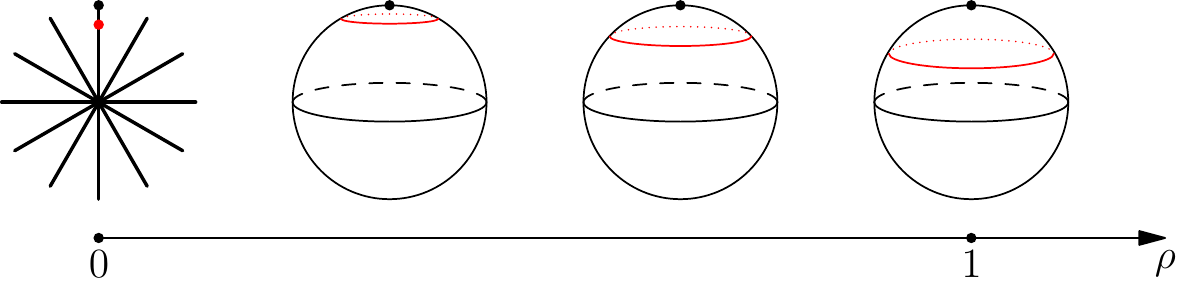}
  \caption{The analytification $\P^{1,\An}$ of the complex projective line with respect to the norm $\|\cdot\|$ on $\C$, together with the canonical map $\lambda\colon\P^{1,\An}\to[0,1]$. The fiber $\lambda^{-1}(0)$ is the analytification of $\P^1$ with respect to the trivial norm, and is homeomorphic to a cone over $\P^1(\C)$. All the other fibers are homeomorphic to a sphere.
The points on top form a continuous section of $\lambda$. The smaller circle in the fiber $\lambda^{-1}(\rho)$ is of radius $e^{-1/\rho}$; these circles converge as $\rho\to0$ to a unique point in the  fiber $\lambda^{-1}(0)$.}\label{F102}
\end{figure}
\begin{ThmC}
  The map $\lambda\colon\XAn\to[0,1]$ is open.
\end{ThmC}
This result essentially says that the Archimedean fibers
$\lambda^{-1}(\rho)$ converge to the 
non-Archimedean fiber $\Xan=\lambda^{-1}(0)$ as $\rho\to0+$.
The latter convergence property implies Theorem~A$'$ since there is a natural
continuous, proper and surjective \emph{tropicalization map} 
$\YAn\to\Ytrop$ that takes
the fiber $\lambda^{-1}(0)$ to $\Xtrop$ and takes any other fiber
$\lambda^{-1}(\rho)$ to the scaled amoeba $\rho\cdot\cA_X$.

We prove Theorem~C using the fact that the points of $\Xan$ of
maximal rational rank are dense. Using resolution of singularities, such
points can be realized on a blowup as monomial valuations with rationally
independent weights, and then the proof is concluded by a direct
computation. See~\S\ref{S108} for details. 
A statement related to Theorem~C appears as Corollary~6.8 in~\cite{PoineauLocale}.

\medskip
Let us make some bibliographical comments.
Amoebae (with the opposite sign convention of ours) were 
introduced in~\cite{GKZ} and have been intensively studied.

When $X=V(f)\subset\C^{*n}$ is a hypersurface, the complement of the amoeba 
$A_X$ in $\R^n$ is convex and its connected components correspond to 
Laurent series expansions of $1/f$ at the origin~\cite{FPT00,PR02,Rul00,Rul01,TdW13}. 
Hypersurface amoebae can also be effectively studied using 
Ronkin functions~\cite{MR01,PR04,Rul01,PT05}. 
Their boundaries are studied in~\cite{MN13a,Mik00,Mik04a,SdW13}

In dimension $n=2$, there is an inequality between the area of 
the amoeba $A_f$ defined by a polynomial $f$ and the area of the 
Newton polygon of $f$~\cite{PR04}: the case of equality was characterized in~\cite{MR01}
as arising from Harnack curves in real algebraic geometry. 
Further interesting relations between real algebraic curves and amoebae are
studied in~\cite{Mik00}. 
The degeneration of amoebae in dimension two onto tropical varieties 
is used in a striking way in~\cite{Mik05} for enumerative problems.
Planar amoebae also arise in certain considerations in statistical thermodynamics~\cite{KOS07,PPT13}.

In higher codimension, amoebae may or may not have finite volume~\cite{MN13b}
but their complements retain certain
weaker convexity properties~\cite{Hen04,Ras09}.
Computational aspects are studied in~\cite{Pur08,The02,TdW11,deW13}.
For more information and further references, see the surveys~\cite{Ite04,Mik04b}.

\smallskip
Tropical varieties have appeared in many different contexts.
We have defined them here as images under the tropicalization map, 
but they can also be characterized in terms of so-called \emph{initial degenerations}~\cite{EKL06,SpeyerSturmfels,Dra08,PayneLimit}.
They have a polyhedral structure~\cite{BieriGroves,EKL06} that satisfies a 
\emph{balancing condition}~\cite{SpeyerThesis,SturmfelsTevelev,GublerGuide}.
Tropical geometry, especially for curves, can also, to some extent,
be developed intrinsically, see~\cite{BN07,BPR11,IMS09,Mik06}.
It has seen striking applications to 
algebraic geometry~\cite{CDPR12,JensenPayne,Mik00,Mik05}.

\smallskip
The relation between Berkovich spaces (over a valued field) 
and tropical geometry appears implicitly already in the work
of Bieri-Groves~\cite{BieriGroves} which predates the general
theory developed by Berkovich himself. Since then, it has been
systematically studied by many authors. 
For finer properties of the tropicalization map, 
see~\eg~\cite{ABBR13,ABBR14,BPR11,DucrosPolytopes,Gub07,GRW14,OssermanPayne,Rab12}.
In~\cite{PayneLimit,FGP13} it is shown that the Berkovich analytification of an algebraic variety 
over a non-Archimedean field is the limit of its tropicalizations over all embeddings into toric varieties.

\smallskip
The general idea of using non-Archimedean techniques to study various kinds
of limiting behavior of complex analytic objects is also not new. 
Morgan and Shalen~\cite{MS84} used valuations to compactify complex
affine varieties. Favre recently used the space $\XAn$ to recast and generalize 
their construction using Berkovich spaces; a statement close to Theorem~C 
(and even closer to Theorem~C$'$ in~\S\ref{S104}) appears 
in~\cite{FavreMS}. 
Other examples of how Berkovich spaces, especially analytifications of 
complex algebraic varieties with respect to the trivial norm, can be
used to study complex analytic phenomena can be found 
in~\cite{BerkHodge,hiro,pshsing,valmul,eigenval,dyncomp,Kiwi1,MS84}.
Also related---at least in spirit---is the procedure known as 
Maslov dequantization: see~\cite{Lit05,IM12} and the references therein.

\smallskip
A version of Theorem~A for a non-Archimedean absolute value 
was proved by Gubler, see~\cite[\S8, Cor.~11.13]{GublerGuide}. 
The techniques in this paper
could likely be adapted to give a new proof of this result, at least
in residue characteristic zero, but we leave this for future work.
It would also be interesting to study the adelic amoebae associated to  
varieties defined over a number field, see~\cite{EKL06,PayneAdelic}. 
The results in this paper have recently been used to study the 
topology of the complements of certain tropical vareities~\cite{NS}.

\medskip
The organization of the paper is as follows. In~\S\ref{S103} we 
recall the notion of continuously varying families of spaces
in the sense of Kuratowski. In~\S\ref{S101} we discuss various 
analytification procedures and prove Theorem~C.
Then, in~\S\ref{S102} we study
the tropicalization map from $\YAn$ to $\Ytrop$ for a toric variety $Y$.
In particular, we prove Theorem~A$'$.
Finally, in~\S\ref{S104} we study one-parameter families of 
varieties and prove Theorem~B, as well as the required fact,
Theorem~C$'$, about Berkovich spaces.

\begin{Ackn}
  I thank V.~Berkovich for the proof of Lemma~\ref{L104},
  and 
  M.~Baker,
  S.~Boucksom, 
  A.~Ducros, 
  W.~Gubler, 
  S.~Payne
  and 
  A.~Werner,
  for comments on a preliminary version
  of this manuscript.
  I have also benefitted from discussions with 
  E.~Brugall{\'e}, C.~Favre, G.~Mikhalkin and B.~Teissier.
\end{Ackn}
%
%
%
%
%
%
\section{Continuous families of subspaces}\label{S103}
Consider a surjective continuous map $\pi:X\to B$ between 
topological spaces. Write $X_b:=\pi^{-1}(b)$ for $b\in B$. 
We'd like to study the continuity properties of $b\mapsto X_b$.
To this end, suppose that $X$ embeds as a subset of $B\times Y$,
for some topological space $Y$, and that $\pi$ is the restriction of 
the projection of $B\times Y$ onto the first factor.
We can then view $X_b$ as a subset of $Y$ for all $b\in B$. 
\begin{Def}\label{defi:lsc}
  We say that $b\to X_b$ is \emph{upper semicontinuous} (usc) if 
  given $b_0\in B$ and $y\in Y\setminus X_{b_0}$, there exist
  neighborhoods $U$ of $y$ in $Y$ and $B_0$ of $b_0$ in $B$ such that 
  $X_b\cap U=\emptyset$ for all $b\in B_0$.
\end{Def}
\begin{Def}\label{defi:usc}
  We say that $b\to X_b$ is \emph{lower semicontinuous} (lsc) if 
  given $b_0\in B$ and $y\in X_{b_0}$ and given any 
  neighborhoods $U$ of $y$ in $Y$ and $B_0$ of $b_0$ in $B$,
  we have $X_b\cap U\ne\emptyset$ for all $b\in B_0$.
\end{Def}
Naturally,  $b\to X_b$ is \emph{continuous} if it is both 
usc and lsc. These continuity properties are in the sense of 
Kuratowski~\cite{Kuratowski}.
The proof of the following result is left to the reader.
\begin{Lemma}\label{L101}
  The map $b\to X_b$ is usc iff $X$ is closed in $B\times Y$.
  It is lsc iff $\pi:X\to B$ is open.
\end{Lemma}
Now suppose $Y$ is a metric space. We can then consider 
continuity of $b\to X_b$ in the \emph{Hausdorff topology},
which means the following:
for every $b_0\in B$ and every $\e>0$ there exists a neighborhood
$B_0$ of $b$ in $B$ such that, whenever $b\in B_0$, 
any point in $X_{b_0}$ (resp. $X_b$) is at distance at most $\e$
from some point in $X_b$ (resp. $X_{b_0}$).

Suppose that $X_b$ is a closed subset of $Y$ for all $b$.
Then continuity of $b\to X_b$ in the Hausdorff topology implies 
continuity in the sense of Kuratowski. The converse is true when
$Y$ is compact.

%
%
%
%
%
%
\section{Analytification}\label{S101}
We recall a special case of the construction in~\cite[\S2]{BerkHodge};
see also~\cite[\S1.5]{BerkBook} and~\cite{PoineauAsterisque,PoineauLocale}.
Consider a Banach ring $(k,\|\cdot\|)$. This means that $k$ is a
commutative ring with unit and that $\|\cdot\|:k\to\R_+$ satisfies
$\|a\|=0$ iff $a=0$; $\|a-b\|\le\|a\|+\|b\|$ and 
$\|ab\|\le\|a\|\cdot\|b\|$ for all $a,b\in k$; and 
$k$ is complete in the metric induced by $\|\cdot\|$.
In fact, we will only consider the case when $k$ is a \emph{field}.

Let $X$ be a separated scheme of finite type over $k$. The construction
in~\cite{BerkHodge} associates an \emph{analytification} $X^{\An}$ of
$X$ with respect to the norm $\|\cdot\|$ on $k$. 
It is defined as follows.\footnote{While we shall only
  consider the analytification as a topological space, one can also
  equip it with a structure sheaf.}
For any affine open subset $U=\Spec A$ of $X$, where $A$ is a 
finitely generated $k$-algebra, let $\UAn$ be the (nonempty) set of
multiplicative seminorms on $A$ whose restrictions to $k$ are bounded
by the norm $\|\cdot\|$.
The topology on $\UAn$ is the weakest one for
which $\UAn\ni|\cdot|\to|f|$ is continuous for every $f\in A$. 

It is customary to denote the points in $\Uan$ by a letter such as
$x$ and the corresponding seminorm by $|\cdot|_x$. 
The latter induces a multiplicative norm on $A/\fp_x$, 
where $\fp_x$ is the kernel of $|\cdot|_x$. Let $\cH(x)$ be the 
completion of the fraction field of $A/\fp_x$ with respect to this norm.

By gluing together the spaces $\UAn$ we construct a 
topological space $\XAn$. This space is Hausdorff, 
locally compact and countable at infinity. The assignment 
$x\mapsto\fp_x$ above globalizes to a continuous map 
\begin{equation*}
  \pi\colon\XAn\to X,
\end{equation*}
where $X$ is viewed as a scheme, equipped with the Zariski topology.
The assignment $X\to\XAn$ is functorial.
If $X\hookrightarrow Y$ is an open (resp.\ closed) embedding, then so is 
$\XAn\hookrightarrow\YAn$. If $X\to Y$ is surjective, then so is
$\XAn\to\YAn$.

The analytification of the zero-dimensional affine
space is equal to the Berkovich spectrum $\cM(k,\|\cdot\|)$ defined
in~\cite[\S1.2]{BerkBook}. The canonical map $X\to\A^0=\Spec k$
induces a surjective, continuous map
\begin{equation*}
  \lambda\colon\XAn\to\cM(k,\|\cdot\|).
\end{equation*}

We shall study this general analytification functor $X\mapsto\XAn$ 
for three types of Banach fields $(k,\|\cdot\|)$.
%
%
%
%
\subsection{Archimedean case}\label{S106}
First assume that $k=\C$ is the field of complex numbers and that 
$\|\cdot\|=|\cdot|_\infty$ is the usual Archimedean norm.
Denote by $X^h$ the usual complex analytic variety associated to 
$X$.
Recall that the points of $X^h$ can be identified with the closed
points of $X$.

It turns out that $X^h$ can be identified with the analytification $\XAn$
above in such a way that $\pi$ maps 
a point of $X^h$ to the corresponding closed point of $X$. 
To see this, first note that
$\cM(\C,|\cdot|_\infty)=\{|\cdot|_\infty\}$ is a singleton.
Now consider an open affine subset $U$.
To each point $x\in U^h$ we can
associate a seminorm $|\cdot|_x\in\UAn$ 
by $|f|_x:=|f(x)|_\infty$. This gives rise to a injective continuous map 
$U^h\to\UAn$ which is surjective by the 
Gelfand-Mazur Theorem,  and easily seen to be a homeomorphism.
%
%
%
%
\subsection{Non-Archimedean case}\label{S105}
Next suppose that $k$ is a \emph{non-Archimedean field}. This means
that $\|\cdot\|=|\cdot|$, where $|\cdot|$ is a non-Archimedean,
multiplicative norm on $k$, that is
$|ab|=|a|\cdot|b|$ and $|a-b|\le\max\{|a|,|b|\}$ 
for any $a,b\in k$.
The analytifications $\XAn$ are then special cases of the 
\emph{Berkovich spaces} studied 
in~\cite{BerkBook,BerkIHES}.\footnote{They are good $k$-analytic spaces without boundary.}.
To conform with the notation in~\loccit we write $\Xan$ instead of $\XAn$.

To any non-Archimedean field $(k,|\cdot|)$ is associated
a \emph{value group}
$|k^\ast|:=\{|a|\ \mid\ a\in k^\ast\}$ 
as well as its divisible version
$\sqrt{|k^\ast|}:=\{r^{1/n}\ \mid\ r\in|k^\ast|,\ n\ge 1\}$.
Now suppose $x\in\Xan$. We can view $\sqrt{|k^\ast|}$
and $\sqrt{|\cH(x)^\ast|}$ as $\Q$-vector spaces. 
Define the \emph{rational rank} $t(x)$ of $x$ as the codimension
of $\sqrt{|k^\ast|}$ in $\sqrt{|\cH(x)^\ast|}$.
If $X$ has dimension $n$, then $t(x)\le n$ for 
all $x\in\Xan$, see~\cite[Lemma~2.5.2]{BerkIHES}.
In fact, $t(x)$ is bounded by the transcendence degree over $k$
of the residue field of $\pi(x)$.
We say that $x$ has \emph{maximal rational rank} 
if $t(x)=n$. In this case, $\pi(x)$ is the generic point of 
an irreducible component of dimension $n$,
and $x$ defines a valuation of the residue field at this point.

Our approach to the proof of Theorem~C in the introduction is based on
\begin{Lemma}\label{L102}
  Assume that $X$ has pure dimension $n$ 
  and that the divisible value group $\sqrt{|k^\ast|}$ 
  has infinite codimension in $\R_+^\ast$ as a $\Q$-vector space.
  Then the set of 
  points in $\Xan$ with maximal rational rank, $t(x)=n$, 
  is dense in $\Xan$.
\end{Lemma}
In fact, a more general statement is true. I am grateful to 
V.~Berkovich for the following statement 
and proof. (Closely related results appear as 
Lemma~10.1.2 of~\cite{DucrosFamilies} and Corollary~5.7 of~\cite{PoineauAngelique}.)
Here we freely use terminology and results 
from~\cite{BerkBook} and~\cite{BerkIHES}.
\begin{Lemma}\label{L104}
  Let $k$ be as in Lemma~\ref{L102}.
  Consider a $k$-analytic space $X$ of pure dimension $n$ 
  and let $X'$ be the set of points $x\in X$ such that $t(x)=n$.
  Then $X'$ is dense in $X$.
\end{Lemma}
\begin{proof}
  We may assume that $X$ is $k$-affinoid. 
  Given positive numbers $r_1,\dots,r_m$ whose images in
  $\R_+^\ast/\sqrt{|k^\ast|}$ are linearly independent, define a 
  valued field extension $K_r/k$ as in~\cite[p.22]{BerkBook}.
  We can pick $r$ such that the base change 
  $Y=X\hat\otimes_kK_r$ is strictly $K_r$-affinoid and 
  of pure dimension $n$. 
  The image of the analogous subset $Y'$ of $Y$ in $X$ under the 
  continuous canonical map $Y\to X$ lies in $X'$. 
  This reduces the situation to the case when $X$ is strictly
  $k$-affinoid and $k$ is nontrivially valued.
  
  The set $X_0$ of points $x\in X$ with $[\cH(x):k]<\infty$ is dense 
  in $X$, and any point of $X_0$ 
  has a fundamental system of strictly affinoid neighborhood,
  see Proposition~2.1.15 and its proof in~\cite{BerkBook}.
  Hence  it suffices to show that every strictly
  $k$-affinoid space of pure dimension $n$ contains a point $x$ with $t(x)=n$. 
  By Noether normalization, the situation is reduced to 
  the case when $X$ is a closed polydisc of radii one. 
  By the assumption on $k$, we can find numbers $0<r_1,\dots,r_n<1$ 
  whose images in $\R_+^\ast/\sqrt{|k^\ast|}$ are linearly independent. 
  Then the maximal point of the closed 
  polydisc of radii $(r_1,\dots,r_n)$ belongs to $X'$.
\end{proof}

Now we specialize to the case when $k=\C$ is the field
of complex numbers and $|\cdot|=|\cdot|_0$ is
the \emph{trivial norm}. Berkovich spaces
over this non-Archimedean field has seen a surprising number
of applications, see for 
example~\cite{BerkHodge,hiro,FavreMS,pshsing,eigenval,dyncomp,dynberko,Thuillier}. 
Their topological structure is partially described in~\cite{hiro,valtree,dynberko}.

Consider a complex algebraic variety $X$ of pure dimension $n$.
Here is an example of a point $x\in\Xan$ of maximal rational rank, $t(x)=n$.
Suppose $\xi\in X$ is a closed point, 
$X$ is smooth at $\xi$ and there exist coordinates $z_1,\dots,z_n$
at $\xi$ and positive numbers $\a_i>0$, $1\le i\le n$.
Then we can define a \emph{monomial valuation} $v$ on 
$\widehat{\scO_{X,\xi}}\simeq\C\cro{z_1,\dots,z_n}$
by setting 
\begin{equation*}
  v(\sum_{m\in\Z_+^n}a_mz^m)
  :=\min\{m_1\a_1+\dots+m_n\a_n\mid a_m\ne0\}.
\end{equation*}
The valuation $v$ defines a point $x=e^{-v}$ in $\Xan$ with $\pi(x)=\xi$,
and we have $t(x)=n$ iff the numbers $\a_i$ are linearly independent
over $\Q$. We call $x$ a \emph{monomial point}.
\begin{Lemma}\label{L103}
  Assume that $X$ has pure dimension $n$ and
  that $x\in\Xan$ has maximal rational rank
  $t(x)=n$.
  Then there exists a surjective birational morphism 
  $\f:Y\to X$, with $Y$ smooth, and a monomial point 
  $y\in\Yan$ with $t(y)=n$ and $\f^{\an}(y)=x$.
\end{Lemma}
\begin{proof}
  The point $x$ defines a real rank one valuation on the function field
  of $X$ and the condition $t(x)=n$ implies that this valuation is an
  Abhyankar valuation. The statement to be proved is then an
  example of local uniformization of Abhyankar valuations,
  see~\cite{KK05}. A simple proof using Hironaka's theorem on 
  resolutions of singularities is given in~\cite[Proposition~2.8]{ELS};
  see also~\cite[Proposition~3.7]{jonmus}.
\end{proof}
%
%
%
%
\subsection{Hybrid case}\label{S107}
Finally we consider the ``hybrid'' construction of~\cite[\S2]{BerkHodge} that
combines Archimedean and non-Archimedean information.
Equip $\C$ with the norm $\|\cdot\|$ defined in~\eqref{eq:norm}, that
is,
\begin{equation*}
  \|\cdot\|:=\max\{|\cdot|_\infty,|\cdot|_0\}.
\end{equation*}

The Berkovich spectrum $\cM(\C,\|\cdot\|)$ is the set
of multiplicative seminorms $|\cdot|$ on $\C$ bounded by $\|\cdot\|$. 
Such a seminorm has to be of the form
$|\cdot|_\infty^\rho$ for some $\rho\in[0,1]$, where the case $\rho=0$
is interpreted as the trivial norm. Thus we can identify
$\cM(\C,\|\cdot\|)$
with the interval $[0,1]$, so we get a surjective, continuous map
\begin{equation*}
  \lambda\colon\XAn\to[0,1].
\end{equation*}
Concretely, this map can be defined by $\lambda(x)=\log|e|_x$.

The fiber $\lambda^{-1}(\rho)$ is equal to the analytification of
$X$ with respect to the multiplicative norm $|\cdot|_\infty^\rho$ 
on $\C$ (where $\rho=0$ is interpreted as the trivial norm.)

In view of~\S\ref{S106}, the fiber $\lambda^{-1}(1)$ is 
therefore homeomorphic to (and will be identified with) $X^h$
in such a way that $\pi$ maps
a point of $X^h$ to the corresponding closed point of $X$.

For $0<\rho\le1$, the fiber $\lambda^{-1}(\rho)$ is also 
homeomorphic to $X^h$: each seminorm $|\cdot|$ in 
$\XAn\cap\lambda^{-1}(\rho)$ is of the form 
$|f|=|f(x)|_\infty^\rho$ for some $x\in X^h$. In fact,
$\lambda^{-1}(]0,1])$ is homeomorphic to the 
product $]0,1]\times X^h$, see~\cite[Lemma~2.1]{BerkHodge}.

Finally, the fiber $\lambda^{-1}(0)$ is the 
Berkovich analytification of $X$ with respect to the trivial norm on
$\C$, as in~\S\ref{S105}. 
Following~\cite{BerkHodge}, we denote this space by $\Xan$.

Any closed point $\eta\in X$ gives rise to a continuous 
section $s_\eta$ of $\lambda$: if $\eta\in U=\Spec A$, then $s_\eta(\rho)$
is the multiplicative seminorm on $A$ defined by 
$f\mapsto|f(\eta)|_\infty^\rho$.

See Figure~4 for a picture of the space $\XAn$ when $X=\P^1$.
%
%
%
%
\subsection{Proof of Theorem~C}\label{S108}
We must prove that $\lambda\colon\XAn\to[0,1]$ is open. 
Recall that there exists a homeomorphism
$\lambda^{-1}(]0,1])\simto\,]0,1]\times X^h$ that commutes with $\lambda$,
so the restriction of 
$\lambda$ to $\XAn\setminus\Xan$ is open. 
Therefore, it suffices to prove that for any $x\in\Xan$,
the pair $(X,x)$ satisfies:

\medskip
$(\star)$ for any neighborhood $U$ of $x$ in $\XAn$, $\lambda(U)$
    is a neighborhood of $0$ in $[0,1]$

\medskip\noindent
In fact, it suffices to prove~$(\star)$ for $x$ of maximal rational rank, 
since by Lemma~\ref{L102} such points are dense in $\Xan$.
Thus assume $t(x)=n$.
By Lemma~\ref{L103} we can find a surjective birational morphism 
$\phi\colon Y\to X$ and a monomial point $y\in\Yan$ such that 
$\phi^{\An}(y)=x$.
Since $\phi^{\An}$ is continuous and surjective, it suffices to prove
$(\star)$ for the pair $(Y,y)$.

Thus we may assume that $X$ is smooth and that $x$ is a 
monomial point. 
By assumption, there exists a closed point $\xi\in X$ such that 
$v=-\log|\cdot|_x$ is a monomial valuation on $\cO_{X,\xi}$
in some local coordinates $z_1,\dots,z_n$ at $\xi$,
say with weights $\a_i=v(z_i)>0$, where $\a_1,\dots,\a_n$ are
linearly independent over $\Q$. 
Upon replacing $X$ by an open affine neighborhood, 
we assume that $X=\Spec A$ is affine and that $z_i\in A$ for all $i$.

For $0<\rho\ll1$, consider the following
polycircle in the coordinates $z_i$
\begin{equation*}
  Z'_\rho=\{\eta\in X^h\ \mid\ |z_i(\eta)|=e^{-\a_i/\rho}
  \quad\text{for $1\le i\le n$}\}.
\end{equation*}
Also write $Z_\rho$ for the image of $Z'_\rho$ under the 
isomorphism $\lambda^{-1}(1)\simto\lambda^{-1}(\rho)$. 
We claim that if $U$ is any neighborhood of $x$ in $\XAn$,
and $0<\e\ll1$, then 
then $Z_\rho\subset U$ for $0<\rho\le\e^2$. This will show that 
$\lambda(U)\supset[0,\e^2]$ and hence complete the proof.

To prove the claim, we may assume that $U$ is of the form
\begin{equation*}
  U^+(f,t):=\{y\in\XAn\mid|f|_y<t\}
  \quad\text{or}\quad
  U^-(f,t):=\{y\in\XAn\mid|f|_y>t\}
\end{equation*}
where $f\in A$ and $t>0$.
Indeed, finite intersections of such sets form a basis of neighborhoods
of $x$ in $\XAn$.
We consider only the case $U=U^+(f,t)$, 
leaving the case $U=U^-(f,t)$ to the reader.
Pick a real number $s>0$ such that 
\begin{equation*}
  |f|_x<s<t
\end{equation*}
Expand $f$ as a power series 
\begin{equation*}
  f=\sum_{m\in\Z_+^n}a_mz^m
\end{equation*}
in $\widehat{\cO_{X,\xi}}\simeq\C\cro{z_1,\dots,z_n}$.
This series converges in some neighborhood of $\xi$ in 
$X^h$, so there exists $R\ge1$ such that 
\begin{equation}\label{e101}
  |a_m|_\infty\le R^{|m|}
\end{equation}
for all $m$, where we write $|m|=m_1+\dots+m_n$.

Since the $\a_i$ are rationally independent, there exists
$\bm\in\Z_+^n$ such that $a_{\bm}\ne 0$ and 
$\langle \bm,\a\rangle<\langle m,\a\rangle:=\sum_{i=1}^n m_i\a_i$ 
for all $m\ne \bm$ such that $a_m\ne 0$.
Note that $e^{-\langle\bm,\a\rangle}=|f|_x<s$.
We choose $\e$ small enough so that if $0<\rho\le\e^2$, then
\begin{equation}\label{e102}
  R^{\rho|\bm|}\le\sqrt{\frac{t}{s}},
\end{equation}
\begin{equation}\label{e103}
  R^{\rho|m|} e^{-\langle m,\a\rangle}
  \le R^{\rho|\bm|}e^{-\langle \bm,\a\rangle-\e|m|}
  \quad\text{when $a_m\ne 0$ and $m\ne\bm$},
\end{equation}
and
\begin{equation}\label{e104}
  \left(\sum_{m\in\Z_+^m}e^{-|m|/\e}\right)^\rho<\sqrt{\frac{t}{s}}.
\end{equation}
  
We claim that $Z_\rho\subset U$ for $0<\rho\le\e^2$ 
for such $\e$.
To see this, pick $y\in Z_\rho$.
We use~\eqref{e101}-\eqref{e104} 
to estimate the terms in the series expansion of $f$.
First,
\begin{equation*}
  |a_{\bm}z(\eta)^{\bm}|_\infty
  =|a_{\bm}|_\infty\cdot|z^{\bm}|_y^{1/\rho}
  \le R^{|\bm|}e^{-\langle\bm,\a\rangle/\rho}.
\end{equation*}
Second, if $m\ne\bm$ and $a_m\ne0$, then
\begin{multline*}
  |a_mz(\eta)^m|_\infty
  =|a_m|_\infty\cdot|z^m|_y^{1/\rho}
  \le R^{|m|}e^{-\langle m,\a\rangle/\rho}\\
  \le R^{|\bm|}e^{-\langle\bm,\a\rangle/\rho}e^{-\e|m|/\rho}
  \le R^{|\bm|}e^{-\langle\bm,\a\rangle/\rho}e^{-|m|/\e}.
\end{multline*}
Since $m\ne 0$ when $m\ne\bm$ and $a_m\ne0$, this leads to 
\begin{multline*}
  |f|_y
  =|f(\eta)|_\infty^\rho
  \le\left(\sum_m|a_mz(\eta)^m|_\infty\right)^\rho\\
  \le\left(R^{|\bm|}e^{-\langle\bm,\a\rangle/\rho}\right)^\rho
  \left(1+\sum_{m\ne\bm, a_m\ne0}e^{-|m|/\e}\right)^\rho\\
  \le R^{|\bm|\rho}e^{-\langle\bm,\a\rangle}
  \left(\sum_{m\in\R_+^n}e^{-|m|/\e}\right)^\rho
  <\sqrt{\frac{t}{s}}\cdot s\cdot \sqrt{\frac{t}{s}}
  =t,
\end{multline*}
and hence $y\in U=U^+(f,t)$, completing the proof.
%
%
%
%
%
%
\section{Toric varieties and tropicalization}~\label{S102}
We recall some basic definitions about toric varieties from~\cite{FultonToric}.
Let $N\simeq\Z^n$ be a lattice, $M=\Hom(N,\Z)$ the dual
lattice, and $\Sigma$ a fan in $N$.
To each cone $\sigma\in\Sigma$ is associated a finitely generated 
monoid $S_\sigma:=\check{\sigma}\cap M$, a finitely generated
algebra $\Z[S_\sigma]$ and an affine variety 
$U_\sigma=\Spec\Z[S_\sigma]$. By suitably gluing together the different affine varieties
$U_\sigma$ over $\sigma\in\Sigma$, we obtain a toric variety $Y=Y_\Sigma$.

We can also associate a tropical object $\Ytrop=Y_\Sigma^\mathrm{trop}$ 
to $\Sigma$ following~\cite[\S4.1]{FultonToric} or~\cite{AMRT};
see also~\cite{Kajiwara} or~\cite{PayneLimit}.\footnote{As with the 
  case of the analytification, the tropicalization 
  $\Ytrop$ will only be considered as a topological space 
  (together with an action by $\R_+^*$) and not equipped with a
  structure sheaf.}
Namely, consider the additive monoid $\Rbar:=\R\cup\{+\infty\}$
equipped with the natural topology. For each cone $\sigma\in\Sigma$, let 
$U^{\trop}_\sigma=\Hom(S_\sigma,\Rbar)$ be the set of
monoid homomorphisms, and equip $U_\sigma^{\trop}$ with 
the topology of pointwise convergence.
For example, $U^{\trop}_0=N_\R:=N\otimes_\Z\R\simeq\R^n$.
The space $\Ytrop$ is obtained by gluing together $U_\sigma^{\trop}$
for $\sigma\in\Sigma$ and contains $N_\R$ as an open dense subset. 
It comes with the scaling action by $\R_+^*$ induced by the same
action on $\Rbar$.
For a polarized projective toric variety $Y$, the moment map
gives a homeomorphism of $\Ytrop$ onto the 
moment polytope in $M_\R=M\otimes_\Z\R$.
%
%
%
%
\subsection{Tropicalization}
As in~\S\ref{S101}, let $\YAn$ be the analytification of $Y\times_\Z\C$ 
with respect to the norm $\|\cdot\|$ on $\C$. We have a continuous map
\begin{equation*}
  \trop\colon\YAn\to\Ytrop
\end{equation*}
defined as follows.
Let $\sigma$ be a cone in $\Sigma$. 
A point in $U_\sigma^{\An}$ is a multiplicative seminorm $|\cdot|$ on 
$\C[S_\sigma]$ whose restriction to $\C$ is bounded by $\|\cdot\|$.
In particular, $-\log|\cdot|$ defines a monoid homomorphism
from $S_\sigma$ to $\Rbar$, and hence an element in $U_\sigma^{\trop}$.
It is easy to verify that the maps 
$U_\sigma^{\An}\to U_\sigma^{\trop}$
glue together to a globally defined continuous map 
$\trop\colon\YAn\to\Ytrop$.

Let $\lambda\colon\YAn\to[0,1]$ be the canonical map, and set 
\begin{equation*}
  \YTrop:=[0,1]\times\Ytrop
  \qand
  \Trop:=\lambda\times\trop.
\end{equation*}
This leads to a commutative diagram
\begin{equation*}
  \xymatrix{%
    \YAn \ar[dr]_{\lambda}  \ar[r]^{\Trop} & \YTrop \ar[d]\\
    & [0,1] }
\end{equation*}
where the map $\YTrop\to[0,1]$ is the 
projection onto the first factor.
\begin{Prop}\label{P101}
  For any toric variety $Y$, the map 
  \begin{equation*}
  \Trop\colon\YAn\to\YTrop
  \end{equation*}
  is continuous, proper and surjective; hence it is also closed.
\end{Prop}
\begin{proof}
  We basically argue as in Lemma~2.1 and~\S3 of~\cite{PayneLimit},
  but include some details as our setting is slightly different.
  The statements to be proved are local on either the source or target, 
  so it suffices to consider the case when $Y=U_\sigma$ is affine.

  In this case, the continuity of the map $\C[S_\sigma]^{\An}\to U_\sigma^{\Trop}$
  is clear from the definition. To prove properness, pick generators 
  $m_1,\dots,m_N$ of the monoid $S_\sigma$. 
  It suffices to prove that if $0\le\rho\le\rho'\le 1$ 
  and $-\infty<s_i\le t_i\le+\infty$ for $1\le i\le N$,
  then the set 
  \begin{equation*}
    W:=\Trop^{-1}\left([\rho,\rho']\times\{v\in U_\sigma^{\trop}\mid 
      s_i\le v(m_i)\le t_i\ \text{for $1\le i\le N$}\}\right)
  \end{equation*}
  is compact in $\C[S_\sigma]^{\An}$.
  Now, the characters
  $z_i:=\chi^{m_i}$, $1\le i\le N$ generate $\C[S_\sigma]$ as a 
  $\C$-algebra; we have 
  \begin{equation*}
    \C[S_\sigma]\simeq\C[z_1,\dots,z_N]/\fa
  \end{equation*}
  for some (monomial) ideal $\fa\subset\C[z_1,\dots,z_N]$.
  Under this identification, $W$ becomes the set of 
  multiplicative seminorms $|\cdot|$ on $\C[z_1,\dots,z_N]$
  whose restrictions to $\C$ are bounded by $\|\cdot\|$, and 
  such that $e^\rho\le|e|\le e^{\rho'}$,
  $e^{-t_i}\le|z_i|\le e^{-s_i}$ for $1\le i\le N$,
  and $|f|=0$ for all $f\in\fa$.
  It is then clear that $W$ is compact, 
  as a consequence of Tychonoff's Theorem.

  Finally, surjectivity can be established as follows. 
  Pick any $(\rho,v)\in U_\sigma^{\Trop}$
  and let $m_i$, $1\le i\le N$, be generators of $S_\sigma$ as before.
  Set $t_i:=v(m_i)\in\Rbar$.

  First suppose $\rho=0$. Define a multiplicative 
  seminorm $|\cdot|$ on $\C[z_1,\dots,z_N]$ by 
  $|\sum_\b a_\b z^\b|=\max\{e^{-\langle t,\b\rangle}\mid a_\b\ne0\}$,
  where $\langle t,\b\rangle=\sum_{i=1}^Nt_i\b_i$.
  This seminorm vanishes on the ideal $\fa$, and hence
  induces a multiplicative seminorm $|\cdot|$ on $\C[S_\sigma]$ 
  whose restriction to $\C$ is the trivial norm. 
  It is then clear that $\Trop(|\cdot|)=(0,v)$.

  Now suppose $0<\rho\le 1$. Let  $\eta\in\Spec\C[z_1,\dots,z_N]$
  be the closed point with coordinates $z_i(\eta)=e^{-t_i}$, $1\le i\le N$,
  and define a multiplicative 
  seminorm $|\cdot|$ on $\C[z_1,\dots,z_N]$
  by $|f|=|f(\eta)|_\infty^\rho$.
  As before, this induces a multiplicative seminorm on 
  $\C[U_\sigma]$ whose restriction to $\C$ is equal to $|\cdot|_\infty^\rho$,
  so $\Trop(|\cdot|)=(\rho,v)$.
\end{proof}
%
%
%
%
\subsection{Proof of Theorem~A$'$}
Let $X$ be a complex algebraic subvariety of $Y\times_\Z\C$.
Then $\XAn$ is a closed subset of $\YAn$.
Let $\Xtrop\subset\Ytrop$ and $\XTrop\subset\YTrop$ be the
images of $\XAn$ under the mappings $\trop$ and $\Trop$,
respectively. 
By Proposition~\ref{P101}, $\XTrop$ is closed in $\YTrop$.
We have a commutative diagram
\begin{equation*}
  \xymatrix{%
    \XAn \ar[dr]_{\lambda}  \ar[r]^{\Trop} & \XTrop \ar[d]^{\pi_1}\\
    & [0,1] }
\end{equation*}
The map $\lambda\colon\XAn\to[0,1]$ is continuous and surjective,
and by Theorem~C it is also open.
The map $\Trop\colon\XAn\to\XTrop$ is surjective by definition and
continuous by Proposition~\ref{P101}. 
It follows from these two properties that 
$\pi_1\colon\XTrop\to[0,1]$ is open and surjective.

Write $\pi_1^{-1}(\rho)=\{\rho\}\times X_\rho^{\trop}$ for 
$0\le\rho\le 1$, where $X_\rho^{\trop}\subset\Ytrop$.  
Lemma~\ref{L101} implies that 
$\rho\mapsto X_\rho^{\trop}$ is continuous.
Theorem~A$'$ will thus follow immediately if we 
can prove that 
$X_0^{\trop}=\Xtrop$ and $X_\rho^{\trop}=\rho\cdot A_X$
for $0<\rho\le1$.

Now, the fiber $\lambda^{-1}(1)$ of $\XAn$ is the analytification
of $X$ with respect to the Archimedean norm $|\cdot|_\infty$ on $\C$.
Hence the fiber $X_1^{\trop}$ of $\XTrop$ is equal to the 
amoeba $A_X$. 
Similarly, for $0<\rho\le1$, $\lambda^{-1}(\rho)$ is the 
analytification of $X$ with respect to the norm $|\cdot|_\infty^\rho$ 
on $\C$, and this implies that $X_\rho$ is the scaled amoeba
$\rho\cdot A_X$ for $0<\rho\le1$.
Finally, the fiber $X_0^{\trop}$ is the image of $\Xan\subset\Yan$
under the tropicalization map $\Yan\to\Ytrop$, where the analytifications
are defined using the trivial norm on $\C$. This image is equal
to the tropicalization $\Xtrop$ of $X$ as defined in~\cite{GublerGuide}.
We should check that this image also agrees with the definition of $\Xtrop$ in
the introduction. On the one hand, the tropicalization does not 
change under non-Archimedean field extensions, 
see~\cite[Prop.~3.7]{GublerGuide}. On the other hand, $\Xan$ may
be viewed as the set of equivalence classes of $L$-valued points, over all
valued field extensions $(L,|\cdot|)$ of $(\C,|\cdot|_0)$, 
see~\cite[3.4.2]{BerkBook}.
This completes the proof.
%
%
%
%
%
%
\section{One-parameter families}\label{S104}
Consider a complex algebraic variety $\cX$ that admits a surjective 
morphism 
\begin{equation*}
  p\colon\cX\to\G_m
\end{equation*}
where $\G_m=\Spec\C[t^{\pm1}]\simeq\C^*$.
We can view $\cX$ as a one-parameter family of complex algebraic
varieties, and we are interested in the behavior as $t\to0$.

As in~\S\ref{S107}, let $\cXAn$ be the Berkovich analytification with respect 
to the norm $\|\cdot\|$ on $\C$, and consider the closed subset 
$\cXsharp\subset\cXAn$ of seminorms for which $|t|=e^{-1}$.
The morphism $p$ gives rise to a continuous surjective map 
$p^{\An}\colon\cXAn\to\G_m^{\An}$ 
that sends $\cXsharp$ to $\G_m^{\sharp}$,
and is equivariant with respect to the continuous maps
$\lambda\colon\cXAn\to[0,1]$ and $\lambda\colon\G_m^{\An}\to[0,1]$.
Write $X_\rho^{\sharp}$ (resp. $\G_{m,\rho}^{\sharp}$) for the fiber 
$\lambda^{-1}(\rho)$ inside $\cXsharp$ (resp. $\G_m^{\sharp}$).

Note that $\G_{m,\rho}^{\sharp}$ consists of all multiplicative 
seminorms $|\cdot|$ on $\C[t^{\pm1}]$ such that $|t|=e^{-1}$ and 
$|a|=|a|_\infty^\rho$ for all $a\in\C^*$. 
In particular, $\G_{m,0}^{\sharp}$ is a singleton, consisting of
the restriction to $\C[t^{\pm1}]$ of the multiplicative non-Archimedean norm
on $\C\lau{t}$ such that $|t|=e^{-1}$ and $|a|=1$ for $a\in\C^*$.
Now let $0<\rho\le 1$. Any seminorm $|\cdot|$ in $\G_{m,\rho}^{\sharp}$
is then of the form $|f|:=|f(a)|_\infty^\rho$ for some $a\in\C^*$,
and the condition $|t|=e^{-1}$ means exactly that $|a|_\infty=e^{-1/\rho}$.
Thus $\G_{m,\rho}^{\sharp}$ is in bijection with the circle of radius $e^{-1/\rho}$ 
in $\C$, so $\G_m^{\sharp}$ can and will be identified with the closed disc $\Delta_{e^{-1}}$ 
of radius $e^{-1}$ in $\C$. Under this identification we have 
\begin{equation*}
  \lambda(a)=\left(\log|a|^{-1}\right)^{-1}
  \quad\text{for $a\in\Delta_{e^{-1}}$}.
\end{equation*}
Write $p^\sharp\colon\cXsharp\to\Delta_{e^{-1}}$ for the restriction 
of $p^{\An}$ to $\cXsharp$, and
$X_a^{\sharp}$ for the fiber above $a\in\Delta_{e^{-1}}$. 
The central fiber $X_0^{\sharp}$ is isomorphic to the
analytification of the base change $\cX\times_{\G_m}\C\lau{t}$,
with respect to the non-Archimedean norm on $\C\lau{t}$.
Any other fiber $X_a^{\sharp}$, $0<|a|\le e^{-1}$, is homeomorphic to the 
fiber above $t=a$ of the complex analytic space $\cX^h$.
\begin{ThmCp}
  For $0<\delta\ll1$, the
  map $p^\sharp\colon\cXsharp\to\Delta_{e^{-1}}$ is open above $\Delta_\delta$.
\end{ThmCp}
\begin{Rmk}
  One can check that when $\cX=\G_m\times X$ is a product, 
  Theorem~C$'$ implies Theorem~C in the introduction.
\end{Rmk}
\begin{proof}
  Using Hironaka's theorem on resolution of singularities, we can find 
  a proper and surjective birational morphism $\cY\to\cX$, with $\cY$ smooth.
  Then $p^\sharp\colon\cY^{\sharp}\to\Delta_{e^{-1}}$ factors through a continuous
  surjective map $\cY^{\sharp}\to\cX^{\sharp}$.
  Hence, if $\cY^{\sharp}\to\Delta_\delta$ is open for some $\delta\in(0,1)$, 
  then so is $\cX^{\sharp}\to\Delta_\delta$. 
  We may therefore assume that $\cX$ is smooth.

  There exists a finite subset $A\subset\G_m$ 
  such that $p\colon\cX\to\G_m$ is flat above $\G_m\setminus A$,
  see for example~\cite[Ch.~III, Prop.~9.7]{Hartshorne}.
  By~\cite[Corollary, p.~73]{Douady} this implies that 
  $p^h\colon\cX^h\to\C^*$, the analytification of $p$ with respect to 
  $|\cdot|_\infty$, is open above $\C^*\setminus A$.
  Pick $\delta>0$ small enough so that $|a|>\delta$ for all $a\in A$.
  Then $p^\sharp\colon\cX^\sharp\to\Delta_{e^{-1}}$ is open above 
  $\Delta_\delta\setminus\{0\}$.
  It remains to see that $p^\sharp$ is open also at points on the 
  non-Archimedean fiber $X^\sharp_0$.

  By the Nagata compactification theorem (see~\cite{ConradNagata}) 
  there exists a proper complex algebraic variety $\ocX$,
  and an open immersion $\cX\hookrightarrow\ocX$, with dense image,
  such that $p$ extends to a proper morphism $p\colon\ocX\to\P^1$.
  Using resolution of singularities, we may assume that $\ocX$ is smooth.
  Again by~\cite[Ch.~III, Prop.~9.7]{Hartshorne},  
  $p\colon\ocX\to\P^1$ is automatically flat above $\P^1\setminus A$. 

  The general properties of the analytification functor imply that 
  $\cX^{\An}$ is an open subset of $\ocX^{\An}$.
  We need to show that if $x\in X^{\sharp}_0$ and $U$ is a neighborhood of
  $x$ in $\cX^{\sharp}$, then $p(U)$ is a neighborhood of $0$ in $\Delta_{e^{-1}}$.
  Since the $\Q$-vector space $\sqrt{|\C\lau{t}^\ast|}$ is of dimension one,
  Lemma~\ref{L104} applies. We may therefore assume that $x$ is a point of 
  maximal rational rank, $t(x)=n$, since such points are dense in
  $X^{\sharp}_0$.
  The divisible value group $\sqrt{|\cH(x)^\ast|}$ of $x$ is
  a $\Q$-vector space of dimension $n+1$; hence $x$ defines
  an Abhyankar valuation of the function field of $\cX$ of 
  rational rank $n+1$. 
  That advantage of having $\ocX$ proper and smooth is now that this
  valuation admits a unique \emph{center} on $\ocX$, 
  as a consequence of the valuative criterion of
  properness. The center is a point $\xi\in\ocX_0$ such that the valuation
  is nonnegative on the local ring $\cO_{\ocX,\xi}$ and strictly positive on the maximal ideal.

  Using~\cite[Proposition~3.7]{jonmus}
  or~\cite[Remark~3.8]{siminag} we may, after a suitable blowup of
  $\ocX$ above $0\in\P^1$, assume that there exist
  local coordinates $z_1,\dots,z_{n+1}$
  at $\xi$, positive integers $b_1,\dots,b_{n+1}$ and rationally
  independent positive real numbers $\a_1,\dots,\a_{n+1}$ such that 
  $t=u\prod_{i=1}^{n+1}z_i^{b_i}$, with $u$ a unit in $\cO_{\ocX,\xi}$, 
  and the point $x$ defines a monomial valuation $v$ on $\cO_{\ocX,\xi}$
  in these coordinates, with values $v(z_i)=\a_i$ for $1\le i\le n+1$. 
  In particular, $\sum_{i=1}^{n+1}b_i\a_i=v(t)=1$.

  For $0<|a|\ll1$, set
  \begin{equation*}
    Z_a:=X^{\sharp}_a\cap\{|z_i|=e^{-\a_i}\ \text{for $1\le i\le n+1$}\}.
  \end{equation*}
  The same type of estimates as in the proof of Theorem~C now show 
  that any open neighborhood $U$ of $x$ in $\cXsharp$ will contain 
  $Z_a$ for $0<|a|\ll1$. 
  Indeed, we may assume $U=\{|f|>t\}$ or $U=\{|f|<t\}$,
  where $t>0$ and $f\in\cO_{\ocX,\xi}$.
  This proves that $p^\sharp(U)$ is an open neighborhood of
  $0$ in $\Delta_\delta$,  as was to be shown.
\end{proof}
%
%
%
%
\subsection{Proof of Theorem~B}
The product $\G_m\times Y$
is a toric variety, and we have $(\G_m\times Y)^{\trop}=\R\times\Ytrop$.
The image of $(\G_m\times Y)^{\sharp}$ in 
$(\G_m\times Y)^{\trop}$ is given by 
\begin{equation*}
  \trop((\G_m\times Y)^{\sharp})
  =\{1\}\times\Ytrop
  \simeq\Ytrop.
\end{equation*}
Via the identification $\G_m^{\sharp}\simeq\Delta:=\Delta_{e^{-1}}$ above,
this induces a commutative diagram
\begin{equation*}
  \xymatrix{%
    (\G_m\times Y)^{\sharp} \ar[dr]_{p}  \ar[r]^{p\times\trop} 
    & \Delta\times\Ytrop\ar[d]\\
    & \Delta }
\end{equation*}

Now suppose $\cX$ is a closed subvariety of 
$(\G_m\times Y)\times_\Z\C\simeq\C^*\times(Y\times_\Z\C)$ 
such that the projection of 
$\cX$ to $\C^*$ is surjective.
Let $\cX^{\dagger}\subset\Delta\times\Ytrop$ be 
the image of $\cX^{\sharp}$ under $p\times\trop$.
Its fiber over $a\in\Delta$ is then equal to 
\begin{equation*}
  X^{\dagger}_a:=\trop(X^{\sharp}_a),
\end{equation*}
where, again, $X^{\sharp}_a=p^{-1}(a)$.
If $a\ne 0$, then $X^{\dagger}_a=\lambda(a)\cdot A_{X_a}$.
On the other hand, $X^{\dagger}_0$ is equal to $\cX^{\trop}$, the 
image of $X^{\sharp}_0$ in $\Ytrop$.
This completes the proof.
%
%
%
%
%
%

\end{document}

%% file: preamb.tex
\numberwithin{equation}{section}       

\theoremstyle{plain}
\newtheorem{Thm}{Theorem}[section]
\newtheorem{Prop}[Thm]{Proposition}
\newtheorem{Lemma}[Thm]{Lemma}

\newtheorem{Prop-def}[Thm]{Proposition-Definition}

\newtheorem*{ThmA}{Theorem A} 
\newtheorem*{ThmAp}{Theorem A$'$} 
\newtheorem*{ThmB}{Theorem B} 
 
\newtheorem*{ThmC}{Theorem C} 
\newtheorem*{ThmCp}{Theorem C$'$}

\theoremstyle{definition}

\newtheorem{Def}[Thm]{Definition}
\newtheorem*{Ackn}{Acknowledgment}

\theoremstyle{remark}

\newtheorem{Rmk}[Thm]{Remark}

\newcounter{alphasect}
\def\alphainsection{0}

\let\oldsection=\section
\def\section{%
  \ifnum\alphainsection=1%
\e    \addtocounter{alphasect}{1}
  \fi%
\oldsection}%

\renewcommand\thesection{%
 \ifnum\alphainsection=1%
   \Alph{alphasect}%
 \else
   \arabic{section}%
 \fi%
}%

\newcommand{\A}{{\mathbf{A}}}
\newcommand{\C}{{\mathbf{C}}}

\newcommand{\G}{{\mathbf{G}}}

\renewcommand{\P}{{\mathbf{P}}}
\newcommand{\Q}{{\mathbf{Q}}}
\newcommand{\R}{{\mathbf{R}}}
\newcommand{\Z}{{\mathbf{Z}}}

\newcommand{\Rbar}{\overline{\R}}

\newcommand{\bm}{{\bar{m}}}

\newcommand{\scO}{{\mathscr{O}}}

\newcommand{\cA}{{\mathcal{A}}}

\newcommand{\cH}{{\mathcal{H}}}

\newcommand{\cL}{{\mathcal{L}}}
\newcommand{\cM}{{\mathcal{M}}}

\newcommand{\cO}{{\mathcal{O}}}

\newcommand{\cX}{{\mathcal{X}}}
\newcommand{\cY}{{\mathcal{Y}}}

\newcommand{\ocX}{\overline{\mathcal{X}}}

\newcommand{\fa}{{\mathfrak{a}}}

\newcommand{\fp}{{\mathfrak{p}}}

\newcommand{\simto}{\overset\sim\to}

\renewcommand{\a}{\alpha}
\renewcommand{\b}{\beta}

\newcommand{\e}{\varepsilon}
\newcommand{\f}{\varphi}

\newcommand{\eg}{e.g.\ }

\newcommand{\loccit}{\textit{loc.\,cit.\ }}

\newcommand{\qand}{{\quad\text{and}\quad}}

\newcommand{\Uan}{{U^\mathrm{an}}}
\newcommand{\UAn}{{U^\mathrm{An}}}

\newcommand{\Xan}{{X^\mathrm{an}}}
\newcommand{\XAn}{{X^\mathrm{An}}}

\newcommand{\Xtrop}{{X^\mathrm{trop}}}
\newcommand{\XTrop}{{X^\mathrm{Trop}}}

\newcommand{\cXAn}{{\mathcal{X}^\mathrm{An}}}
\newcommand{\cXsharp}{{\mathcal{X}^{\sharp}}}
\newcommand{\cXtrop}{{\mathcal{X}^\mathrm{trop}}}

\newcommand{\Yan}{{Y^\mathrm{an}}}
\newcommand{\YAn}{{Y^\mathrm{An}}}
\newcommand{\Ytrop}{{Y^\mathrm{trop}}}
\newcommand{\YTrop}{{Y^\mathrm{Trop}}}

\newcommand{\an}{\operatorname{an}}
\newcommand{\An}{\operatorname{An}}

\newcommand{\Hom}{\operatorname{Hom}}

\newcommand{\Spec}{\operatorname{Spec}}

\newcommand{\trop}{\operatorname{trop}}
\newcommand{\Trop}{\operatorname{Trop}}

\newcommand{\cro}[1]{[\![#1]\!]}
\newcommand{\lau}[1]{(\!(#1)\!)}

%% file: amoebaev2.bbl
\begin{thebibliography}{BJSST07}

\bibitem[ABBR13]{ABBR13}
O.~Amini, M.~Baker, E.~Brugall{\'e} and J.~Rabinoff.
\newblock\emph{Lifting harmonic morphisms I: metrized complexes and Berkovich skeleta}.
\newblock\texttt{arXiv:1303.4812}.
\newblock To appear in Res. Math. Sci.

\bibitem[ABBR14]{ABBR14}
O.~Amini, M.~Baker, E.~Brugall{\'e} and J.~Rabinoff.
\newblock\emph{Lifting harmonic morphisms II: tropical curves and metrized complexes}.
\newblock Algebra Number Theory \textbf{9} (2015), 267--315.

\bibitem[AMRT]{AMRT}
A.~Ash, D.~Mumford,~M.~Rapoport and Y.-S.~Tai.
\newblock \emph{Smooth Compactifications of Locally Symmetric Varieties}.
\newblock  Second Edition. With the collaboration of Peter Scholze.
\newblock Cambridge University Press, Cambridge, 2010.

\bibitem[AKNR13]{AKNR13}
M.~Avenda{\~n}o, R.~Kogan, M.~Nisse and J.~M.~Rojas.
\newblock\emph{Metric estimates and membership complexity for Archimedean amoebae and tropical hypersurfaces}.
\newblock\texttt{arXiv:1307.3681}.

\bibitem[BN07]{BN07}
M.~Baker and S.~Norine.
\newblock\emph{Riemann-Roch and Abel-Jacobi theory on a finite graph}.
\newblock Adv. Math. \textbf{215} (2007), 766--788.

\bibitem[BPR11]{BPR11}
M.~Baker, S.~Payne and J.~Rabinoff.
\newblock\emph{Nonarchimedean geometry, tropicalization, and metrics on curves}.
\newblock\texttt{arXiv:1104.0320}.

\bibitem[Berg71]{Bergman}
G.~M.~Bergman.
\newblock\emph{The logarithmic limit-set of an algebraic variety}.
\newblock Trans. Amer. Math. Soc. \textbf{157} (1971) 459--469.

\bibitem[Berk]{BerkBook}
V.~Berkovich.
\newblock \emph{Spectral theory and analytic geometry over non-Archimedean fields}.
\newblock Mathematical Surveys and Monographs, vol. 33.
\newblock American Mathematical Society, Providence, RI, 1990.

\bibitem[Berk93]{BerkIHES}
V.~G.~Berkovich.
\newblock \emph{\'Etale cohomology for non-Archimedean analytic spaces}.
\newblock Publ. Math. Inst. Hautes \'Etudes Sci. \textbf{78} (1993), 5--161. 

\bibitem[Berk09]{BerkHodge}
V.~G.~Berkovich.
\newblock \emph{A non-Archimedean interpretation of the weight zero
  subspaces of limit mixed Hodge structures}.
\newblock In \emph{Algebra, arithmetic, and geometry: in honor of Yu.~I.~Manin}.
\newblock Progr. Math., vol 269, 49--67.
\newblock Birkh\"auser, Boston, MA, 2009.

\bibitem[BG84]{BieriGroves}
R.~Bieri and J.~R.~R.~Groves.
\newblock\emph{The geometry of the set of characters induced by valuations}.
\newblock J. Reine Angew. Math. \textbf{347} (1984), 168--195.

\bibitem[BJSST07]{BJSST}
T.~Bogart, A.~N.~Jensen, D.~Speyer,~B.~Sturmfels and R.~R.~Thomas.
\newblock\emph{Computing tropical varieties}.
\newblock J. Symbolic Comput. \textbf{42} (2007), 54--73.

\bibitem[BFJ08]{hiro}
S.~Boucksom, C.~Favre and M.~Jonsson.
\newblock \emph{Valuations and plurisubharmonic singularities}.
\newblock Publ. Res. Inst. Math. Sci.  44  (2008), 449--494.

\bibitem[BFJ12]{siminag} 
S.~Boucksom, C.~Favre and M.~Jonsson.
\newblock \emph{Singular semipositive metrics in non-Archimedean geometry}.
\newblock \texttt{arXiv:1201.0187}. To appear in J. Algebraic Geom.

\bibitem[CP12]{CartwrightPayne}
D.~Cartwright and S.~Payne.
\newblock \emph{Connectivity of tropicalizations}.
\newblock Math. Res. Lett. \textbf{19} (2012), 1089--1095.

\bibitem[Con07]{ConradNagata}
B.~Conrad.
\newblock \emph{Deligne's notes on Nagata compactifications}.
\newblock J. Ramanujan Math. Soc. \textbf{22} (2007), 205--257. 

\bibitem[CDPR12]{CDPR12}
F.~Cools, J.~Draisma, S.~Payne and E.~Robeva.
\newblock\emph{A tropical proof of the Brill-Noether Theorem}.
\newblock Adv. Math. \textbf{230} (2012), 759--776.

\bibitem[Dou68]{Douady}
A.~Douady.
\newblock \emph{Flatness and privilege}.
\newblock Ens. Math. \textbf{14} (1968), 47--74.

\bibitem[Dra08]{Dra08}
J.~Draisma.
\newblock \emph{A tropical approach to secant dimensions}.
\newblock J. Pure Appl. Algebra \textbf{212} (2008), 349--363.

\bibitem[Duc11]{DucrosFamilies}
A.~Ducros.
\newblock \emph{Families of Berkovich spaces}.
\newblock \texttt{arXiv:1107.4259v3}.

\bibitem[Duc12]{DucrosPolytopes}
A.~Ducros.
\newblock \emph{Espaces de Berkovich, polytopes, squelettes et th\'eorie des mod\`eles}.
\newblock Confluentes Math. \textbf{4}, No. 4 (2012), 1250007.

\bibitem[ELS03]{ELS}
L.~Ein, R.~Lazarsfeld and K.~Smith.
\newblock \emph{Uniform approximation of Abhyankar valuation ideals in smooth
function fields}.
\newblock  Amer. J. Math. \textbf{125} (2003), 409--440.

\bibitem[EKL06]{EKL06}
M.~Einsiedler, M.~Kapranov and D.~Lind.
\newblock\emph{Non-archimedean amoebas and tropical varieties}.
\newblock J. Reine Angew. Math. \textbf{601} (2006), 139--157.

\bibitem[Fav12]{FavreMS}
C.~Favre.
\newblock\emph{Compactifications of affine varieties with non-Archimedean techniques and dynamical applications}.
\newblock Preliminary manuscript, 2012.

\bibitem[FJ04]{valtree}
C.~Favre and M.~Jonsson.
\newblock \emph{The valuative tree}.  
\newblock  Lecture Notes in Mathematics, vol. 1853.
\newblock Springer-Verlag, Berlin, 2004. 

\bibitem[FJ05a]{pshsing}
C.~Favre and M.~Jonsson.
\newblock \emph{Valuative analysis of planar pluri\-sub\-harmonic functions}.
\newblock  Invent. Math.  162  (2005),  no. 2, 271--311. 

\bibitem[FJ05b]{valmul}
C.~Favre and M.~Jonsson.
\newblock \emph{Valuations and multiplier ideals}.
\newblock J. Amer. Math. Soc, \textbf{18} (2005), 655--684.

\bibitem[FJ07]{eigenval}
C.~Favre and M.~Jonsson.
\newblock \emph{Eigenvaluations}.
\newblock Ann. Sci. {\'E}cole Norm. Sup. \textbf{40} (2007), 309--349. 

\bibitem[FJ11]{dyncomp}
C.~Favre and M.~Jonsson.
\newblock \emph{Dynamical compactifications of $\C^2$}.
\newblock Ann.\ of Math. \textbf{173} (2011), 211--249.

\bibitem[FPT00]{FPT00}
M.~Forsberg, M.~Passare and A.~Tsikh.
\newblock \emph{Laurent determinants and arrangements of hyperplane amoebas}.
\newblock Adv. Math. \textbf{151} (2000), 45--70.

\bibitem[FGP13]{FGP13}
T.~Foster, P.~Gross and S.~Payne.
\newblock\emph{Limits of tropicalizations}.
\newblock Israel J. Math \textbf{201} (2014), 835--846.

\bibitem[Ful]{FultonToric}
W.~Fulton.
\newblock \emph{Introduction to toric varieties}.
\newblock Annals of Mathematics Studies, 131. 
Princeton University Press, Princeton, NJ, 1993.

\bibitem[GKZ]{GKZ}
I.~M.~Gelfand, M.~M.~Kapranov and A.~V.~Zelevinsky.
\newblock\emph{Discriminants, Resultants, and Multidimensional Determinants}.
\newblock Birkhäuser Boston, Inc., Boston, MA, 1994.

\bibitem[Gub07]{Gub07}
W.~Gubler.
\newblock\emph{Tropical varieties for non-archimedean analytic spaces}.
\newblock Invent. Math. \textbf{169} (2007), 321--376.

\bibitem[Gub13]{GublerGuide}
W.~Gubler.
\newblock\emph{A guide to tropicalizations}.
\newblock In \textit{Algebraic and combinatorial aspects of tropical geometry}, 125--189.
\newblock Contemp. Math., 589.
\newblock Amer. Math. Soc.
\newblock Providence, RI, 2013.

\bibitem[GRW14]{GRW14}
W.~Gubler, J.~Rabinoff and A.~Werner.
\newblock\emph{Skeletons and tropicalizations}.
\newblock\texttt{arXiv:1404.7044}.

\bibitem[Har]{Hartshorne}
R.~Hartshorne. 
\newblock \emph{Algebraic geometry}.
\newblock Graduate Texts in Mathematics, No. 52. 
\newblock Springer-Verlag, New York-Heidelberg, 1977.

\bibitem[Hen04]{Hen04}
A.~Henriques.
\newblock\emph{An analogue of convexity for complements of amoebas of varieties of higher codimension, an answer to a question asked by B. Sturmfels}.
\newblock Adv. Geom. \textbf{4} (2004), 61--73. 

\bibitem[Ite04]{Ite04}
I.~Itenberg.
\newblock\emph{Amibes de vari\'et\'es alg\'ebriques et d\'enombrement de
              courbes (d'apr\`es {G}. {M}ikhalkin)}
\newblock Ast\'erisque, No 294 (2004), 335--361.

\bibitem[IM12]{IM12}
I.~Itenberg and G.~Mikhalkin.
\newblock\emph{Geometry in the tropical limit}.
\newblock Math. Semesterber. \textbf{59} (2012), 57--73.

\bibitem[IMS09]{IMS09}
I.~Itenberg, G.~Mikhalkin and E.~Shustin.
\newblock\emph{Tropical Algebraic Geometry}.
\newblock Oberwolfach Seminars, 35.
\newblock Second edition.
\newblock Birkh\"auser, Basel 2009.

\bibitem[JP14]{JensenPayne}
D.~Jensen and S.~Payne.
\newblock\emph{Tropical independence I: Shapes of divisors and a proof of the 
Gieseker-Petri Theorem}.
\newblock \texttt{arXiv:1401.2584}.
\newblock To appear in Algebra Number Theory.

\bibitem[Jon12]{dynberko}
M.~Jonsson.
\newblock\emph{Dynamics on Berkovich spaces in low dimensions}.
\newblock \texttt{arXiv:1201.1944}.
\newblock In \emph{Berkovich spaces and applications}. 
\newblock Lecture Notes in Mathematics, 2119.
\newblock Editors: A. Ducros, C. Favre, J. Nicaise.
\newblock Springer, 2015.

\bibitem[JM12]{jonmus}
M.~Jonsson and M.~Musta\c{t}\v{a}. 
\newblock \emph{Valuations and asymptotic invariants for sequences of ideals}.
\newblock Ann. Inst. Fourier \textbf{62} (2012), 2145--2209.

\bibitem[Kaj08]{Kajiwara}
T.~Kajiwara.
\newblock\emph{Tropical toric geometry}.
\newblock In \textit{Toric topology}, 197--207.
\newblock Contemp. Math., 460.
\newblock Amer. Math. Soc.
\newblock Providence, RI, 2008.

\bibitem[KOS07]{KOS07}
R.~Kenyon, A.~Okounkov and S.~Sheffield.
\newblock\emph{Dimers and amoebae}.
\newblock Ann.\ of Math. \textbf{163} (2006), 1019--1056. 

\bibitem[Kiw06]{Kiwi1}
J.~Kiwi.
\newblock \emph{Puiseux series polynomial dynamics and 
  iteration of complex cubic polynomials}.
\newblock Ann. Inst. Fourier \textbf{56} (2006), 1337--1404.

\bibitem[KK05]{KK05}
H.~Knaf and F.-V.~Kuhlmann.
\newblock \emph{Abhyankar places admit local uniformization in any characteristic}.
\newblock Ann. Sci. \'Ecole Norm. Sup. (4) \textbf{38} (2005), 833--846.

\bibitem[Kur]{Kuratowski}
K.~Kuratowski.
\newblock\emph{Topology}, Vol I-II.
\newblock Academic Press, New York-London, 1966.

\bibitem[Lit05]{Lit05}
G.~L.~Litvinov.
\newblock\emph{The Maslov dequantization, idempotent and tropical mathematics: a very brief introduction}.
\newblock In \textit{Idempotent mathematics and mathematical physics}, 1--17.
\newblock Contemp. Math., 377.
\newblock Amer. Math. Soc.
\newblock Providence, RI, 2005.

\bibitem[MS15]{MS15}
D.~Maclagan and B.~Sturmfels.
\newblock\emph{Introduction to Tropical Geometry}.
\newblock Graduate Studies in Mathematics, vol 161.
\newblock Amer. Math. Soc.
\newblock Providence, RI, 2015.

\bibitem[MN13a]{MN13a}
F.~Madani and M.~Nisse.
\newblock\emph{Generalized logarithmic Gauss map and its relation to (co)amoebas}. 
\newblock Math. Nachr. \textbf{286} (2013), 1510--1513. 

\bibitem[MN13b]{MN13b}
F.~Madani and M.~Nisse.
\newblock\emph{On the volume of complex amoebas}.
\newblock Proc. Amer. Math. Soc. \textbf{141} (2013), 1113--1123.

\bibitem[Mik00]{Mik00}
G.~Mikhalkin.
\newblock\emph{Real algebraic curves, moment map and amoebas}.
\newblock Ann.\ of Math. \textbf{151} (2000), 309--326.

\bibitem[Mik04a]{Mik04a}
G.~Mikhalkin.
\newblock\emph{Decomposition into pairs-of-pants for complex algebraic hypersurfaces}.
\newblock Topology. \textbf{43} (2004), 1035--1065.

\bibitem[Mik04b]{Mik04b}
G.~Mikhalkin.
\newblock\emph{Amoebas of algebraic varieties and tropical geometry}.
\newblock In \textit{Different faces of geometry}, 257--300.
\newblock Int. Math. Ser. (N. Y.), Volume 3.
\newblock Kluwer/Plenum, New York, 2004.

\bibitem[Mik05]{Mik05}
G.~Mikhalkin.
\newblock\emph{Enumerative tropical geometry in~$\mathbb{R}^2$}.
\newblock J. Amer. Math. Soc. \textbf{18} (2005), 313--377.

\bibitem[Mik06]{Mik06}
G.~Mikhalkin.
\newblock\emph{Tropical geometry and its applications}.
\newblock International Congress of Mathematicians. Vol. II, 827--852.
\newblock Eur. Math. Soc., Z\"urich, 2006.

\bibitem[MR01]{MR01}
G.~Mikhalkin and H.~Rullg{\aa}rd.
\newblock\emph{Amoebas of maximal area}.
\newblock Internat. Math. Res. Notices, No. 9, 2001, 441--451.

\bibitem[MS84]{MS84}
J.~W.~Morgan and P.~B.~Shalen.
\newblock\emph{Valuations, trees, and degenerations of hyperbolic
  structures, I}.
\newblock Ann.\ of Math. \textbf{120} (1984). 401--476.

\bibitem[NS]{NS}
M.~Nisse and F.~Sottile.
\newblock\emph{Higher convexity for complements of tropical varieties}.
\newblock \texttt{arXiv:1411.7363}.

\bibitem[OP13]{OssermanPayne}
B.~Osserman and S.~Payne.
\newblock\emph{Lifting tropical intersections}.
\newblock Doc. Math. \textbf{18} (2013), 121--175.

\bibitem[PR02]{PR02}
M.~Passare and H.~Rullg{\aa}rd.
\newblock\emph{Multiple Laurent series and polynomial amoebas}.
\newblock Actes des Rencontres d'Analyse Complexe (Poitiers-Futuroscope, 1999), 123--129.
\newblock Atlantique, Poitiers, 2002.

\bibitem[PR04]{PR04}
M.~Passare and H.~Rullg{\aa}rd.
\newblock\emph{Amoebas, Monge-Amp\`ere measure, and triangulations of the Newton polytope}.
\newblock Duke Math. J. \textbf{121} (2004), 481--507.

\bibitem[PPT13]{PPT13}
M.~Passare, D.~Potchekutov and A.~Tsikh.
\newblock\emph{Amoebas of complex hypersurfaces in statistical thermodynamics}.
\newblock Math. Phys. Anal. Geom. \textbf{16} (2013), 89--108.

\bibitem[PT05]{PT05}
M.~Passare and A.~Tsikh.
\newblock\emph{Amoebas: their spines and their contours}.
\newblock In \emph{Idempotent mathematics and mathematical physics}, 275--288.
\newblock Contemp. Math. \textbf{377}.
\newblock Amer. Math. Soc, Providence, RI, 2005.

\bibitem[Pay08]{PayneAdelic}
S.~Payne.
\newblock\emph{Adelic amoebas disjoint from open halfspaces}.
\newblock J. Reine Angew. Math. \textbf{625} (2008), 115--123.

\bibitem[Pay09]{PayneLimit}
S.~Payne.
\newblock\emph{Analytification is the limit of all tropicalizations}.
\newblock Math. Res. Lett. \textbf{16} (2009), 543--556.

\bibitem[Poi10]{PoineauAsterisque}
J.~Poineau.
\newblock\emph{La droite de {B}erkovich sur {$\mathbf{Z}$}}.
\newblock Ast\'erisque \textbf{334} (2010).

\bibitem[Poi13a]{PoineauLocale}
J.~Poineau. 
\newblock\emph{Espaces de Berkovich sur $\mathbf{Z}$: \'etude locale}. 
\newblock Invent. Math. \textbf{194} (2013), 535--590. 

\bibitem[Poi13b]{PoineauAngelique}
J.~Poineau. 
\newblock\emph{Les espaces de Berkovich sont ang{\'e}lique}.
\newblock Bull. Soc. Math. France. \textbf{141} (2013), 267--297. 

\bibitem[Pur08]{Pur08}
K.~Purbhoo.
\newblock\emph{A Nullstellensatz for amoebas}.
\newblock Duke Math. J. \textbf{141} (2008), 407--445.

\bibitem[Rab12]{Rab12}
J.~Rabinoff.
\newblock\emph{Tropical analytic geometry, Newton polygons, and tropical intersections}.
\newblock Adv. Math. \textbf{229} (2012), 3192--3255.

\bibitem[Ras09]{Ras09}
A.~Rashkovskii.
\newblock\emph{A remark on amoebas in higher codimension}.
\newblock In \emph{Analysis and mathematical physics}, 465--471.
\newblock Edited by Bj\"orn Gustafsson and Alexander Vasil'ev.
\newblock Birkh\"auser Verlag, Basel, 2009.

\bibitem[Rul00]{Rul00} 
H.~Rullg{\aa}rd.
\newblock\emph{Stratification des espaces de polyn\^omes de Laurent et la
  structure de leurs amibes}.
\newblock C. R. Acad. Sci. Paris S\'er. I Math. \textbf{331} (2000), 355--358.

\bibitem[Rul01]{Rul01} 
H.~Rullg{\aa}rd.
\newblock\emph{Polynomial amoebas and convexity}.
\newblock Research reports in Mathematics, number 8.
\newblock Stockholm University, 2001.

\bibitem[SdW13]{SdW13}
F.~Schroeter and T.~de Wolff.
\newblock\emph{The boundary of amoebas}.
\newblock \texttt{arXiv:1310.7363}.

\bibitem[Spe05]{SpeyerThesis}
D.~Speyer.
\newblock\emph{Tropical geometry}.
\newblock Thesis, University of California, Berkeley, 2005.

\bibitem[SS04]{SpeyerSturmfels}
D.~Speyer and B.~Sturmfels.
\newblock\emph{The tropical Grassmannian}.
\newblock Adv. Geom. \textbf{4} (2004), 389--411.

\bibitem[ST08]{SturmfelsTevelev}
B.~Sturmfels and J.~Tevelev.
\newblock\emph{Elimination theory for tropical varieties}.
\newblock Math. Res. Lett. \textbf{15} (2008), 543--562.

\bibitem[Tei08]{Tei08}
B.~Teissier.
\newblock\emph{Amibes non archim\'ediennes}.
\newblock In \textit{G\'eom\'etrie tropicale}, 85--114.
\newblock Ed. {\'E}c. Polytech., Palaiseau, 2008. 

\bibitem[The02]{The02}
T.~Theobald.
\newblock\emph{Computing amoebas}.
\newblock Experiment. Math. \textbf{11} (2002), 513--526. 

\bibitem[TdW11]{TdW11}
T.~Theobald and T.~de Wolff.
\newblock\emph{Approximating amoebas and coamoebas by sums of squares}.
\newblock \texttt{arXiv:1101.4114}.
\newblock To appear in Math. Comp.

\bibitem[TdW13]{TdW13}
T.~Theobald and T.~de Wolff.
\newblock\emph{Amoebas of genus at most one}.
\newblock Adv. Math. \textbf{239} (2013), 190--213.

\bibitem[Thu07]{Thuillier}
A.~Thuillier.
\newblock \emph{G\'eom\'etrie toro\"{\i}dale et g\'eom\'etrie analytique non archim\'edienne. Application au type d'homotopie de certains sch\'emas formels}.
\newblock Manuscripta Math. \textbf{123} (2007),  381--451.

\bibitem[Vir06]{Viro}
O.~Viro.
\newblock \emph{Patchworking real algebraic varieties}.
\newblock\texttt{arXiv:0611382}.

\bibitem[deW13]{deW13}
T.~de Wolff.
\newblock\emph{On the geometry, topology and approximation of amoebas}.
\newblock Thesis, Goethe Universit\"at, Frankfurt am Main, 2013.

\end{thebibliography}
